\documentclass[12pt,a4paper,reqno]{article}
\usepackage{amsthm}
\usepackage{newtx}
\usepackage{upref}
\usepackage{enumitem}
\usepackage{authblk}
\usepackage[margin=3cm]{geometry}
\usepackage[%
pdfauthor={Daniel Daners, Jochen Glück, James B Kennedy},%
pdftitle={Positivity Properties of the Dirichlet-to-Neumann Operator on Graphs}%
]{hyperref}
\usepackage{float}
\usepackage{pgfplots}
\tikzset{%
  inode/.style={minimum size=18pt,circle,fill=white,inner sep=1pt,draw,font=\footnotesize},%
  onode/.style={minimum size=18pt,circle,white,fill=black!70,inner sep=1pt,draw,font=\footnotesize},%
  node distance = 2.7cm,%
  auto,%
  nedge/.style={thick},
  iedge/.style={thick},
  oedge/.style={iedge},
  cedge/.style={thick,densely dashed},
}
\tikzset{%
  spectrum/.style={thick},%
  spectrumpos/.style={spectrum,densely dashed},%
  spectrumevpos/.style={spectrum,densely dotted},%
}%
\pgfplotsset{compat=1.18}%
\pgfplotsset{width=0.9\textwidth,height=0.5\textwidth,enlargelimits=false}%
\pgfplotsset{grid style={color=gray},xmajorticks=false,grid=major,extra x tick style={xmajorticks=true}}
\definecolor{tblue}{HTML}{1F77B4}
\definecolor{torange}{HTML}{FF7F0E}
\definecolor{tgreen}{HTML}{2CA02C}
\definecolor{tred}{HTML}{D62728}
\definecolor{tpurple}{HTML}{9467BD}
\definecolor{tbrown}{HTML}{8C564B}
\definecolor{tpink}{HTML}{E377C2}
\definecolor{tgray}{HTML}{7F7F7F}
\definecolor{tolive}{HTML}{BCBD22}
\definecolor{tcyan}{HTML}{17BECF}

\title{Positivity Properties of the Dirichlet-to-Neumann Operator on Graphs}

\author[1]{Daniel Daners}%
\author[2]{Jochen Gl\"uck}%
\author[3]{James B. Kennedy}%
\affil[1]{School of Mathematics and Statistics, The University of Sydney, NSW 2006, Australia\authorcr%
  \nolinkurl{daniel.daners@sydney.edu.au}}%
\affil[2]{Fakultät für Mathematik und Naturwissenschaften, Bergische Universität Wuppertal, Gaußstraße 20, D-42119 Wuppertal, Germany\authorcr%
  \nolinkurl{glueck@uni-wuppertal.de}}%
\affil[3]{Departamento de Matemática, Faculdade de Ciências da Universidade de Lisboa, Campo Grande, Edifício C6, P-1749-016 Lisboa, Portugal\authorcr%
  \nolinkurl{jbkennedy@ciencias.ulisboa.pt}}%

\date{\today}

\newtheorem{theorem}{Theorem}[section]
\newtheorem{lemma}[theorem]{Lemma}
\newtheorem{proposition}[theorem]{Proposition}
\newtheorem{corollary}[theorem]{Corollary}

\theoremstyle{definition}
\newtheorem{definition}[theorem]{Definition}

\newtheorem{example}[theorem]{Example}

\theoremstyle{remark}
\newtheorem{remark}[theorem]{Remark}
\newtheorem{remarks}[theorem]{Remarks}

\numberwithin{equation}{section}
\numberwithin{figure}{section}

\DeclareMathOperator{\im}{im}

\DeclareMathOperator{\spb}{s}
\DeclareMathOperator{\sign}{sign}
\DeclareMathOperator{\trace}{\gamma}

\DeclareMathOperator{\repart}{Re}
\DeclareMathOperator{\impart}{Im}

\DeclareMathOperator{\innerVertex}{i}
\DeclareMathOperator{\outerVertex}{o}
\DeclareMathOperator{\leftVertex}{l}
\DeclareMathOperator{\rightVertex}{r}

\newcommand{\Vi}{{V_{\innerVertex}}}
\newcommand{\Vo}{{V_{\outerVertex}}}
\newcommand{\el}{{e_{\leftVertex}}}
\newcommand{\er}{{e_{\rightVertex}}}

\newcommand{\bbN}{\mathbb{N}}

\newcommand{\bbR}{\mathbb{R}}
\newcommand{\bbQ}{\mathbb{Q}}
\newcommand{\bbC}{\mathbb{C}}

\let\oldthebibliography\thebibliography
\renewcommand\thebibliography[1]{
  \oldthebibliography{#1}
  \setlength{\parskip}{0pt}
  \setlength{\itemsep}{0pt plus 0.3ex}
  \small
}

\begin{document}
\maketitle
\renewcommand{\thefootnote}{}%
\begin{abstract}
  We explore positivity properties of the semigroup generated by the negative of the Dirichlet-to-Neumann operator with real potential $\lambda$, defined on a subset of the vertices of a quantum graph. We show that for rationally independent edge lengths and suitable graph topologies, this semigroup will alternate between being positive, eventually positive without being positive (that is, positive only for sufficiently large times), and not even eventually positive, as $\lambda \to \infty$. For other graph topologies, the semigroup will alternate between being positive and not eventually positive. The topological conditions are related to a \emph{reduced graph} which is a schematic map of the connections between the vertices on which the Dirichlet-to-Neumann operator acts.
\end{abstract}

\section{Introduction and main results}
\label{sec:introduction}
The aim of this work is to investigate positivity properties of the semigroup generated by the Dirichlet-to-Neumann operator on a \emph{quantum graph}. Throughout the paper let $G = (V,E)$ be a \emph{simple graph} (no loops or parallel edges) with a non-empty finite set of vertices $V$ and set of edges $E$. We let $L\colon E \to (0,\infty)$ be a function that assigns a length $L_e$ to each edge $e \in E$, making the graph into a quantum graph, sometimes also called a metric graph. This allows us to consider boundary value problems on the edges of $G$ with boundary conditions on the vertices.

We also introduce a non-empty subset $\Vo \subseteq V$ and its complementary set $\Vi := V \setminus \Vo$. For reasons that will become clear later we refer to $\Vo$ as the set of \emph{outer vertices} and to $\Vi$ as the set of \emph{inner vertices}. The aim of this paper is to study positivity properties of solutions to the following problem with dynamic node conditions on the outer vertices $\Vo$ and \emph{Kirchhoff conditions} on the inner vertices $\Vi$:
\begin{equation}
  \label{eq:dynamic-boundary-value-problem}
  \begin{aligned}
    \Delta_{\max} f(t)+\lambda f(t)            & =0 &  & \text{on }G\times[0,\infty),   \\
    (\nu f(t))                                 & =0 &  & \text{on }\Vi\times[0,\infty), \\
    \frac{\partial}{\partial t}f(t)+(\nu f(t)) & =0 &  & \text{on }\Vo\times[0,\infty), \\
    (\gamma f(0))                              & =x &  & \text{on }\Vo,
  \end{aligned}
\end{equation}
with $0 \leq x \in \bbR^{\Vo}$, where $\lambda\in\mathbb R$, $\Delta_{\max}$ is the maximal Laplace operator on $L^2(G)$, $\trace f$ is the \emph{trace} of $f$ on $V$ and $\nu f$ the \emph{outer unit normal derivative} on $V$, see Section~\ref{sec:dtn-operator} for precise definitions.

By introducing the \emph{Dirichlet-to-Neumann operator} on $\bbC^{\Vo}$ we can reformulate \eqref{eq:dynamic-boundary-value-problem} to fit standard semigroup theory. To construct this operator we start with the Laplace operator $\Delta_{\Vo}$ on $G$ with homogeneous Dirichlet boundary conditions on $\Vo$ and Kirchhoff conditions on $\Vi$; this is known to be a self-adjoint operator on $L^2(G)$ having compact resolvent and spectrum $\sigma(-\Delta_{\Vo})\subseteq (0,\infty)$ (see also Proposition~\ref{prop:simple-properties-of-laplacian} below). In particular, the minimal eigenvalue
\begin{equation}
  \label{eq:lambda-1}
  \lambda_1(-\Delta_{\Vo}):=\min \sigma(-\Delta_{\Vo})>0
\end{equation}
is well defined. It turns out that the problem
\begin{equation}
  \label{eq:boundary-value-problem}
  \begin{aligned}
    \Delta_{\max}f+\lambda f & =0   &  & \text{on }G \\
    (\gamma f)|_\Vo          & = x                   \\
    (\nu f)|_\Vi             & = 0.
  \end{aligned}
\end{equation}
has a unique solution for each $x \in \bbC^\Vo$ if and only if $\lambda\in\mathbb R\setminus\sigma(-\Delta_{\Vo})$; see Proposition~\ref{prop:boundary-value-problem-and-eigenvalue-problem} below. The Dirichlet-to-Neumann operator is the operator
\begin{equation*}
  D_{\lambda,\Vo}\colon \bbC^\Vo \to \bbC^\Vo,\qquad
  x \mapsto (\nu f)|_{\Vo},
\end{equation*}
where $f$ is the unique solution of~\eqref{eq:boundary-value-problem} (cf.\ \cite[Section~3.5]{berkolaiko:13:iqg} or \cite[Chapter~17]{kurasov:24:sgg}). We endow $\bbC^{\Vo}$ with the Euclidean norm and its real part $\bbR^{\Vo}$ with the usual component-wise order. If $\Vo$ has $m$ elements enumerated in a fixed order $\{v_1,\dots,v_m\}$, and the edge between the vertices $v_k$ and $v_j$ is identified with $(k,j)$, then $\bbC^{\Vo}$ can be identified with $\bbC^m$ and $D_{\lambda,\Vo}$ can be represented by an $m\times m$ matrix.

This operator and its associated matrix, along with numerous variants, have been studied heavily in the quantum graph literature in recent years; indeed, it is also sometimes called the \emph{Titchmarsh-Weyl $M$-function} or simply \emph{$M$-function}. It also belongs to the larger class of \emph{Herglotz-Nevanlinna operators}, see for instance \cite[Chapter~18]{kurasov:24:sgg}. The operator is studied for its links to inverse scattering problems (see for instance \cite{kostrykin:06:lmg}, \cite[Notes on Chapter 3]{berkolaiko:13:iqg}, or \cite[Chapter 18]{kurasov:24:sgg}), and related questions of analyticity of eigenvalues and eigenfunctions with respect to various graph parameters \cite{kuchment:19:asd}. It has also found applications to estimates for Laplacian eigenvalues via ``surgery'' methods \cite{kurasov:20:ggs}, as well as to the relationship between nodal domains of eigenfunctions and spectral partitions of the graph (or domain) \cite{berkolaiko:19:nds}. Other applications and results can be found for instance in \cite{behrndt:10:nne,carlson:12:dnm,friedlander:17:dno}.

At any rate, this operator allows us to rewrite Problem~\eqref{eq:dynamic-boundary-value-problem} in the equivalent form
\begin{equation}
  \label{eq:dynamic-boundary-value-problem-abstract}
  \begin{aligned}
    \frac{df}{dt}+D_{\lambda,\Vo}f & =0 &  & \text{in }\Vo\times(0,\infty) \\
    f(0)                           & =x &  & \text{on }\Vo,
  \end{aligned}
\end{equation}
its solution can be written in terms of the matrix semigroup $(e^{-tD_{\lambda,\Vo}})_{t\geq 0}$ as $f(t)=e^{-tD_{\lambda,\Vo}}x$, which is also studied in \cite[Section~6.6.1]{mugnolo:14:sme}. The problem \eqref{eq:dynamic-boundary-value-problem-abstract} is analogous to the Dirichlet-to-Neumann semigroup on domains $\Omega\subseteq\mathbb R^N$. The ``outer vertices'' correspond to the boundary $\partial\Omega$ and the ``interior vertices'' correspond to the interior of $\Omega$, explaining our terminology. The outer vertices are also sometimes called contact vertices (see for instance \cite[Chapter~17]{kurasov:24:sgg}), in keeping with the idea of the Dirichlet-to-Neumann operator as a voltage-to-current operator.

In this paper we are interested in various types of positivity properties of the semigroup $(e^{-tD_{\lambda,\Vo}})_{t\geq 0}$ in function of both $\lambda \in \mathbb R$ and the topology of the graph, including certain notions of \emph{eventual positivity}, as considered in \cite{daners:16:eap,daners:16:eps}. In fact, it was already observed earlier, by one of the current authors \cite{daners:14:nps}, following on from \cite{arendt:12:fei}, that on simple domains such as disks, the Dirichlet-to-Neumann semigroup $e^{-tD_\lambda}$ displays rather intricate positivity properties in function of $\lambda$: for some $\lambda > 0$ sufficiently large, even above the first eigenvalue of the Dirichlet Laplacian, it may be positive (that is, $f \geq 0$ implies $e^{-tD_{\lambda}}f \geq 0$); for other large $\lambda$ it may not be positive but rather eventually positive (that is, roughly speaking, $f > 0$ implies $e^{-tD_{\lambda}}f > 0$ for large but not necessarily for small $t>0$); or it may not have any such properties. This, in turn, was one of the examples which inspired the construction of a more general theory of eventual positivity starting in \cite{daners:16:eap,daners:16:eps}. This subject has since seen a rapid development by various authors, for instance in spectral analysis \cite{arora:24:saa, mui:23:lep}, long-term behaviour \cite{arnold:23:gre, vogt:22:sep}, perturbation theory \cite{arora:25:sos, pappu:25:epp}, and in applications to various differential equations on graphs \cite{gregorio::20:blg} and domains \cite{addona:22:bkt, kunze:25:eon, ploss:24:efo}. However, it remains an open question whether the same kind of eventual positivity of the Dirichlet-to-Neumann semigroup observed on the disk also holds on general bounded domains in $\mathbb R^n$.

The current paper is, in some sense, a sequel to works such as \cite{daners:14:nps,daners:16:eps}, where we will extend these considerations to quantum graphs. On the one hand, on a graph the Dirichlet-to-Neumann semigroup reduces to a matrix, making it far more tractable. In fact, we note that there is a large literature studying what patterns of positive entries in a square matrix lead to an eventually positive matrix, see for instance \cite{berman:10:spe,johnson:22:tsp, noutsos:08:rhn} and, for the same phenomenon with general cones in $\mathbb R^n$, \cite{glueck:23:eci, kasigwa:17:eci, sootla:19:pep}. However, as we will see, semigroups arising from Dirichlet-to-Neumann matrices on quantum graphs are a special class: they still display the intricate behaviour observed on the disk, and moreover, this depends crucially on the topology of the graph. To explain this, and to be able to formulate our main results, we first need to introduce a (discrete) graph associated with $G$, with vertices $\Vo$.
\begin{definition}[Reduced graph]
  \label{def:reduced-graph}
  Let $G$ be a graph as defined above and let $\Vo$ and $\Vi$ be the sets of outer and inner vertices, respectively. We let $G[\Vo]$ and $G[\Vi]$ be the corresponding induced sub-graphs of $G$. We define the graph $G_{\outerVertex}$ to be the graph with vertex set $\Vo$ and a single edge between the vertices $v,w\in\Vo$ if there is an edge in $G$ connecting them, or if there is a path through $\Vi$ connecting them. Denote the set of edges by $E_{\outerVertex}$. We will call $G_{\outerVertex}=(\Vo,E_{\outerVertex})$ the \emph{reduced graph} of $G$ associated with $\Vo$.
\end{definition}
We note that the reduced graph of a connected graph is a simple connected graph and that $G_{\outerVertex}=G$ if $\Vo=V$. As an example consider the graph in Figure~\ref{fig:graph-with-subgraphs}. Its reduced graph shown in Figure~\ref{fig:graph-with-subgraphs-reduced}. Dashed lines indicate the cases where vertices are connected through $\Vi$. Where two vertices in $\Vo$ are connected directly \emph{and} through $\Vi$, we have marked both lines to make this clear; however, this is always to be understood as a single edge.
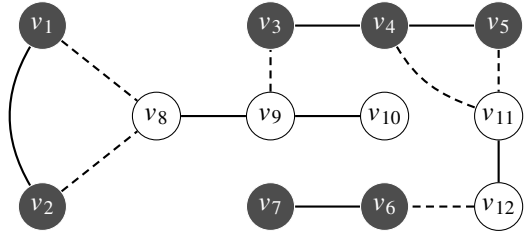
\begin{figure}[ht]
  \centering
  \begin{tikzpicture}[declare function={dv=1.2;dh=1.5;}]
    \node[onode] (v1) at (0,dv) {$v_1$};%
    \node[onode] (v2) at (0,-dv) {$v_2$};%
    \node[onode] (v3) at (2*dh,dv) {$v_3$};%
    \node[onode] (v4) at (3*dh,dv) {$v_4$};%
    \node[onode] (v5) at (4*dh,dv) {$v_5$};%
    \node[onode] (v6) at (3*dh,-dv) {$v_6$};%
    \node[onode] (v7) at (2*dh,-dv) {$v_7$};%
    \node[inode] (w8) at (dh,0) {$v_8$};%
    \node[inode] (w9) at (2*dh,0) {$v_9$};%
    \node[inode] (w10) at (3*dh,0) {$v_{10}$};%
    \node[inode] (w11) at (4*dh,0) {$v_{11}$};%
    \node[inode] (w12) at (4*dh,-dv) {$v_{12}$};%

    \draw (v1) edge[oedge,out=-120,in=120] (v2);%
    \path (v1) edge[cedge] (w8);%
    \path (v2) edge[cedge] (w8);%
    \path (w8) edge[iedge] (w9);%
    \path (w9) edge[iedge] (w10);%
    \path (w9) edge[cedge] (v3);%
    \path (v3) edge[oedge] (v4);%
    \path (v4) edge[oedge] (v5);%
    \path (v5) edge[cedge] (w11);%
    \path (w11) edge[iedge] (w12);%
    \path (v4) edge[cedge,out=-60,in=160] (w11);%
    \path (w12) edge[cedge] (v6);%
    \path (v6) edge[oedge] (v7);%
  \end{tikzpicture}
  \caption{The graph $G$ and connected components of $G[\Vo]$ (vertices in black) and $G[\Vi]$ (vertices in white) with connection between $G[\Vi]$ and $G[\Vo]$ shown as dashed lines.}
  \label{fig:graph-with-subgraphs}
\end{figure}
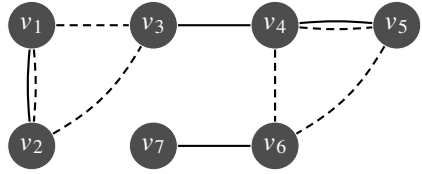
\begin{figure}[ht]
  \centering
  \begin{tikzpicture}[declare function={dv=0.8;dh=1.6;}]
    \node[onode] (v1) at (dh,dv) {$v_1$};%
    \node[onode] (v2) at (dh,-dv) {$v_2$};%
    \node[onode] (v3) at (2*dh,dv) {$v_3$};%
    \node[onode] (v4) at (3*dh,dv) {$v_4$};%
    \node[onode] (v5) at (4*dh,dv) {$v_5$};%
    \node[onode] (v6) at (3*dh,-dv) {$v_6$};%
    \node[onode] (v7) at (2*dh,-dv) {$v_7$};%

    \path (v1) edge[oedge,out=-95,in=95] (v2);%
    \path (v1) edge[cedge,out=-85,in=85] (v2);%
    \path (v3) edge[oedge] (v4);%
    \path (v4) edge[oedge,out=5,in=175] (v5);%
    \path (v4) edge[cedge,out=-5,in=185] (v5);%
    \path (v6) edge[oedge] (v7);%
    \path (v1) edge[cedge] (v3);%
    \path (v2) edge[cedge,out=30,in=-120] (v3);%
    \path (v4) edge[cedge] (v6);%
    \path (v5) edge[cedge,out=-120,in=30] (v6);%
  \end{tikzpicture}
  \caption{The reduced graph $G_{\outerVertex}$ of the graph $G$ from Figure~\ref{fig:graph-with-subgraphs} with connections through $\Vi$ shown as dashed lines. Note that $v_1$ and $v_2$ are considered to be connected by a single edge, likewise $v_4$ and $v_5$.}
  \label{fig:graph-with-subgraphs-reduced}
\end{figure}
\par
The significance of the reduced graph is that it determines the non-zero entries of $D_{\lambda,\Vo}$. For a matrix $Q\in\bbC^{n\times n}$ we write
\begin{equation}
  \label{eq:adjacency-class}
  Q\in\mathcal A(G)\iff Q_{kj}=0\text{ whenever }(k,j)\not\in E\text{ for }k\neq j.
\end{equation}
This means that $Q$ can only have non-zero off-diagonal entries in those places where the adjacency matrix $A_G$ of $G$, given by
\begin{equation}
  \label{eq:adjacency-matrix}
  [A_G]_{kj}:=
  \begin{cases}
    1 & \text{ if }(k,j)\in E, \\
    0 & \text{otherwise,}
  \end{cases}
\end{equation}
has a non-zero entry. In particular, $A_G\in\mathcal A(G)$.
\begin{proposition}
  \label{prop:main}
  Let $D_{\lambda,\Vo}$ be the Dirichlet-to-Neumann operator for the quantum graph $G$ with outer vertex set $\Vo$. Let $G_{\outerVertex}$ be the reduced graph associated with $\Vo$ and let $\lambda\in\mathbb C\setminus\sigma(-\Delta_{\Vo})$. Then the following assertions hold.
  \begin{enumerate}[label={\normalfont (\roman*)}]
  \item\label{item:main:adjacency-class} $D_{\lambda,\Vo}\in\mathcal A(G_{\outerVertex})$.
  \item\label{item:main:formula-for-dtn-operator-with-inner-vertices} Enumerate the vertices such that $V=\{v_k\colon k=1,\dots,n\}$ and $\Vo=\{v_k\colon k=1,\dots,m\}$. If $m<n$ write
    \begin{equation}
      \label{eq:DN-V}
      -D_{\lambda,V} =
      \begin{bmatrix}
        A & B^T \\
        B & C
      \end{bmatrix}
      ,
    \end{equation}
    where $A$ corresponds to the first $m$ nodes $\{v_1,\dots,v_m\} = \Vo$. More precisely, $A\in\bbC^{m \times m}$, $B\in\bbC^{m \times (n-m)}$, $B^T \in \bbC^{(n-m) \times m}$ is the transpose of $B$ and $C \in \bbC^{(n-m) \times (n-m)}$. Then, $C$ is invertible and $-D_{\lambda,\Vo}$ is given by the \emph{Schur complement} of $C$, that is,
    \begin{equation}
      \label{eq:DN-Vo}
      -D_{\lambda,\Vo} = A + B^T(-C)^{-1}B.
    \end{equation}
  \end{enumerate}
\end{proposition}
The proposition shows in particular that the part $B^T(-C)^{-1}B$ contributes to the entry $[D_{\lambda,\Vo}]_{rs}$ if and only if there is a connection between $v_r$ and $v_s$ through $\Vi$, and that $A_{rs}$ contributes if and only if there is a direct connection between $v_r$ and $v_s$ in $\Vo$. Where these contributions occur is visualised in Figure~\ref{fig:graph-with-subgraphs-reduced} by dashed and solid lines, respectively.

The formula \eqref{eq:DN-Vo} is known in the field, see for instance \cite[Theorem~17.9]{kurasov:24:sgg} or \cite[Section~3.5]{berkolaiko:13:iqg}; it can also be found in \cite[Lemma~3.1]{kennedy:20:eqg}, which is based on an early draft of the present article. However, since we cannot find a complete proof of the full statement elsewhere, we will give a proof below. Together with assertion \ref{item:main:adjacency-class}, formula \eqref{eq:DN-Vo} is an important ingredient in some of our results on the positivity of the semigroup generated by $-D_{\lambda,\Vo}$, which we can now introduce.

We call the semigroup $(e^{-tD_{\lambda,\Vo}})_{t\geq 0}$ \emph{strongly positive} if all entries of the matrix $e^{-tD_{\lambda,\Vo}}$ are positive for all $t>0$. We call the semigroup $(e^{-tD_{\lambda,\Vo}})_{t\geq 0}$ \emph{eventually strongly positive} if there exists $t_0>0$ such that all entries of the matrix $e^{-tD_{\lambda,\Vo}}$ are positive for all $t\geq t_0$. Our main result is the following theorem.
\begin{theorem}
  \label{thm:main-result}
  Let $G$ be a connected quantum graph as introduced above, $\Vi$ and $\Vo$ the sets of inner and outer vertices, respectively, and let $G_{\outerVertex}$ be the reduced graph associated with $\Vo$. Then the following assertions hold.
  \begin{enumerate}[label={\normalfont (\roman*)}]
  \item \label{thm:main:basic}If $\lambda < \lambda_1(-\Delta_{\Vo})$, then the semigroup $(e^{-tD_{\lambda, \Vo}})_{t \ge 0}$ is strongly positive.
  \item \label{thm:main:tree}If the reduced graph $G_{\outerVertex}$ is a tree, then for every $\lambda\in\mathbb R\setminus\sigma(-\Delta_{\Vo})$, the semigroup $(e^{-tD_{\lambda, \Vo}})_{t \ge 0}$ is either positive or not eventually strongly positive.
  \end{enumerate}
  Suppose in addition that the family $(L_e)_{e\in E}$ of edge lengths is linearly independent over $\mathbb Q$. Then the following assertions hold.
  \begin{enumerate}[label={\normalfont (\roman*)}, resume]
  \item\label{thm:main:positive} For each $\hat \lambda \in \bbR$ there exists $\lambda > \hat \lambda$ such that the semigroup $(e^{-tD_{\lambda, \Vo}})_{t \ge 0}$ is strongly positive.
  \item\label{thm:main:nonpos} If $|\Vo| \ge 2$, then, for each $\hat \lambda \in \bbR$, there exists $\lambda > \hat \lambda$ such that the semigroup $(e^{-tD_{\lambda, \Vo}})_{t \ge 0}$ is not eventually positive.
  \item\label{thm:main:evpos} If the reduced graph $G_{\outerVertex}$ has a cycle, then, for each $\hat \lambda \in \bbR$, there exists $\lambda > \hat \lambda$ such that the semigroup $(e^{-tD_{\lambda, \Vo}})_{t \ge 0}$ is eventually strongly positive, but not positive.
  \end{enumerate}
\end{theorem}
We note that by definition, a cycle is a closed path in the graph. Such a path traverses each vertex and edge exactly once. As we work with simple graphs a cycle thus requires that $|\Vo|\geq 3$. We also note that, as a consequence of Lemma~\ref{lem:eventually-positive-semigroups} below, the set of $\lambda\in\mathbb R$ where the semigroup is eventually strongly positive is open.

As discussed above, the results on positivity are inspired by those of the Dirichlet-to-Neumann semigroup on domains from~\cite{arendt:12:fei} and in particular those on the disk from~\cite{daners:14:nps} for large positive values of $\lambda$. See also \cite[Section~6.2]{daners:16:eps} and \cite[Section~6]{daners:16:eap}.
\begin{remarks}
  \label{rem:various}
  (a) The condition in Theorem~\ref{thm:main-result}\ref{thm:main:positive}--\ref{thm:main:evpos} that the lengths of the edges be linearly independent cannot be dropped. If, for instance, there are no inner vertices and the lengths of all edges coincide, then it follows from Proposition~\ref{prop:explicit-formula-for-dtn-operator} below that, for each $\lambda \in \bbR \setminus \sigma(-\Delta_V)$, all off-diagonal entries of $D_{\lambda,V} = D_{\lambda,\Vo}$ have the same sign. If this sign is positive, then the semigroup $(e^{-tD_{\lambda,V}})_{t \ge 0}$ is strongly positive (see also Theorem~\ref{thm:main-commensurable}). If, on the other hand, the sign is negative, then the semigroup cannot even be eventually positive according to Lemma~\ref{lem:positive-matrix-group} below (since $D_{\lambda,V}$ is not a diagonal matrix). Hence, assertion~\ref{thm:main:evpos} of Theorem~\ref{thm:main-result} is not fulfilled in this case.

  (b) Theorem~\ref{thm:main-result}\ref{thm:main:nonpos} fails if we do not assume that $|\Vo| \geq 2$ since every real semigroup in dimension one is strongly positive.

  (c) In Theorem~\ref{thm:main-result}\ref{thm:main:evpos}, the condition that the reduced graph have a cycle is essential. A graph without a cycle by definition is a tree, and for such cases Theorem~\ref{thm:main-result}\ref{thm:main:tree} shows that eventual strong positivity without positivity is not possible. The reason is that for the semigroup to be strongly eventually positive but not positive, we require at least one negative off-diagonal entry. However, a minimal pattern of positive off-diagonal elements in $-D_{\lambda,\Vo}$ is required for eventual strong positivity. Changing the sign of any of them will make the semigroup not eventually positive. The theorem shows that such minimal patterns are given by trees. This phenomenon is illustrated in Examples~\ref{ex:evpos-counterexample} and~\ref{ex:star-graph}.

  (d) The assumption in Theorem~\ref{thm:main-result}\ref{thm:main:evpos} is never satisfied if $|\Vo|=1$, and the statement is not true if $|\Vo| = 2$ since in dimension two a semigroup is either positive or not eventually positive, see \cite[Proposition~6.2]{daners:16:eps}. A counterexample is also provided by the graph consisting of two vertices and one edge between them, that is, the Dirichlet-to-Neumann operator on an interval, see \cite[Section~3]{daners:14:nps}.
\end{remarks}
We can also say something about positivity when the family $(L_e)_{e\in E}$ of edge lengths is commensurable, that is, there exists a family $(n_e)_{e\in E}$ and an $L>0$ such that $n_eL=L_e$ for all $e\in E$.
\begin{theorem}
  \label{thm:main-commensurable}
  Suppose that the family of edge lengths $(L_e)_{e\in E}$ is commensurable. Then there exists $L>0$ such that $(e^{-tD_{\lambda,\Vo}})_{t>0}$ is strongly positive for all
  \begin{equation}
    \label{eq:main-commensurable}
    \lambda\in\left\{\left(\sqrt{\mu}+\dfrac{2\pi p}{L}\right)^2\colon \mu\in\left(0,\lambda_1(-\Delta_{\Vo})\right),\,p\in\mathbb N\right\}.
  \end{equation}
\end{theorem}

We will start by giving a collection of examples illustrating the different possibilities in Section~\ref{sec:examples}. The formal construction of the Dirichlet-to-Neumann operator then follows, in Section~\ref{sec:dtn-operator}. Even though known, we include complete proofs since it provides the tools to prove Theorem~\ref{thm:main-result}\ref{thm:main:basic}, as well as to make the paper more self-contained. In Section~\ref{sec:matrix-representation} we establish the matrix representation of the Dirichlet-to-Neumann operator, including a proof of Proposition~\ref{prop:main}. In Section~\ref{sec:limit-theorems} we establish the existence of sequences of $\lambda$ that allow to control the entries of $D_{\lambda,\Vo}$ in a very precise way. These results are the key for proving Theorems~\ref{thm:main-result} and~\ref{thm:main-commensurable} in Section~\ref{sec:positivity}. In Appendix~\ref{sec:matrix-semigroups} we establish some facts on positive and eventually positive matrix semigroups.

\section{Examples}
\label{sec:examples}
In this section we provide some examples, supported by numerical calculations. The plots of the spectra were numerically created using the Python module \emph{Numpy}\footnote{\url{https://numpy.org}}. The operators $D_{\lambda,\Vo}$ were computed using the Schur complement formula from Proposition~\ref{prop:main}\ref{item:main:formula-for-dtn-operator-with-inner-vertices} together with Proposition~\ref{prop:explicit-formula-for-dtn-operator}, both derived in Section~\ref{sec:matrix-representation} (see Example~\ref{ex:nw-4-3-matrix} for an example). Positivity was tested by checking that the off-diagonal elements of $-D_{\lambda,\Vo}$ are non-negative. The eventual strong positivity was numerically tested by checking that the dominant eigenfunction has components that are all of the same sign. Since $D_{\lambda,\Vo}$ is symmetric, the spectral projection associated with the spectral bound is strongly positive in that case. Hence the semigroup is strongly eventually positive, but not positive by \cite[Theorem~5.4]{daners:16:eps}, provided some off-diagonal entries of $-D_{\lambda,\Vo}$ are negative.

The simplest example is the a graph with two nodes and one connection, that is, an interval. That semigroup was investigated in detail in \cite[Section~3]{daners:14:nps}, see also Figure~\ref{fig:graph-2-2}. In that example, positivity and non-positivity of the semigroup generated by the Dirichlet-to-Neumann operator alternate at every eigenvalue of the Dirichlet Laplace operator. They are given by $\lambda_k=\left(\dfrac{\pi k}{L}\right)^2$, where $k\in\bbN$.
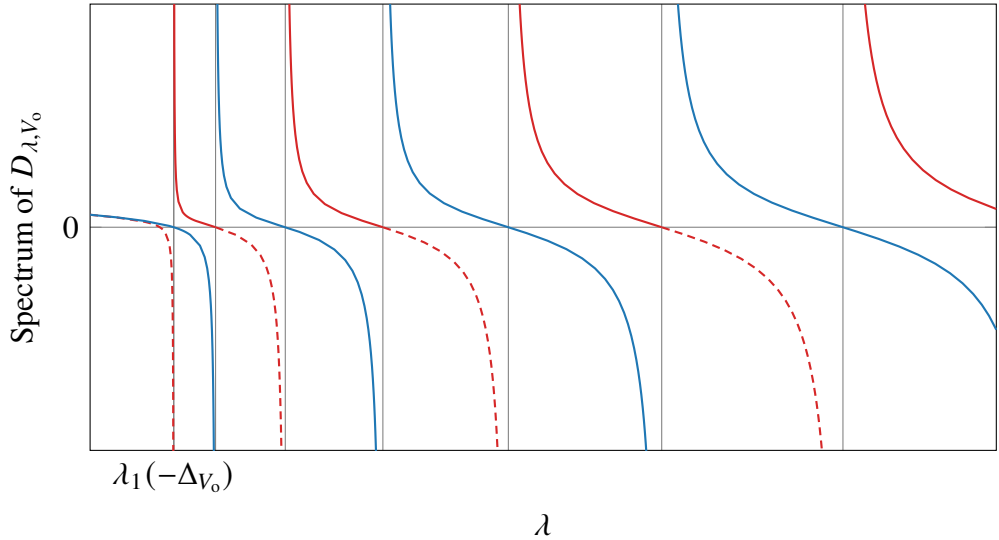
\begin{figure}[ht]
  \centering
  \begin{tikzpicture}
    \begin{axis}[%
        xmin=-5,%
        xmax=60,%
        ymin=-40,%
        ymax=40,%
        xtick={1,4,9,16,25,36,49},%
        ytick={0},%
        extra x ticks={1},
        extra x tick labels={$\lambda_1(-\Delta_{\Vo})$},
        xlabel={$\lambda$},%
        ylabel={Spectrum of $D_{\lambda,\Vo}$},%
      ]%
      \begin{scope}[tred]
        \addplot[spectrum] coordinates {%
            (1.020,64.610)(1.040,32.772)(1.060,22.155)(1.080,16.843)(1.100,13.653)(1.120,11.524)(1.160,8.857)(1.220,6.666)(1.340,4.581)(1.640,2.715)(1.660,2.647)(1.680,2.582)(1.700,2.521)(1.720,2.463)(1.740,2.408)(1.760,2.355)(1.780,2.304)(1.800,2.255)(1.820,2.209)(1.840,2.164)(1.860,2.121)(1.880,2.079)(1.900,2.039)(1.920,2.001)(1.940,1.963)(1.960,1.927)(1.980,1.892)(2,1.858)(2.020,1.825)(2.040,1.793)(2.060,1.761)(2.080,1.731)(2.100,1.701)(2.120,1.672)(2.140,1.644)(2.160,1.617)(2.180,1.590)(2.200,1.563)(2.220,1.538)(2.240,1.512)(2.260,1.488)(2.280,1.463)(2.300,1.440)(2.320,1.416)(2.340,1.393)(2.360,1.371)(2.380,1.348)(2.400,1.327)(2.420,1.305)(2.440,1.284)(2.460,1.263)(2.480,1.242)(2.500,1.222)(2.520,1.202)(2.540,1.182)(2.560,1.162)(2.580,1.143)(2.600,1.124)(4,0)};%
        \addplot[spectrum] coordinates {%
            (9.280,41.804)(9.300,39.070)(9.320,36.677)(9.340,34.565)(9.380,31.007)(9.420,28.124)(9.460,25.740)(9.500,23.736)(9.560,21.264)(9.640,18.684)(9.720,16.672)(9.840,14.364)(10,12.136)(10.220,10.006)(10.540,7.960)(11.040,5.985)(11.840,4.147)(13.040,2.529)(14.520,1.169)(16,0)};%
        \addplot[spectrum] coordinates {%
            (25.800,40.530)(25.860,37.738)(25.920,35.308)(25.980,33.173)(26.060,30.700)(26.140,28.571)(26.240,26.292)(26.360,23.994)(26.500,21.771)(26.660,19.682)(26.860,17.564)(27.100,15.541)(27.400,13.562)(27.800,11.553)(28.320,9.621)(29.020,7.742)(29.960,5.948)(31.140,4.324)(32.520,2.876)(34.040,1.553)(35.620,0.298)(36,0)};%
        \addplot[spectrum] coordinates {%
            (50.520,41.596)(50.600,39.522)(50.700,37.202)(50.800,35.136)(50.920,32.937)(51.060,30.691)(51.220,28.464)(51.400,26.308)(51.600,24.253)(51.840,22.159)(52.120,20.110)(52.460,18.051)(52.860,16.069)(53.360,14.074)(53.980,12.113)(54.740,10.229)(55.660,8.447)(56.760,6.771)(58.040,5.206)(60,3.266)};%
      \end{scope}
      \begin{scope}[tred]
        \addplot[spectrumpos] coordinates {%
            (-5,2.232)(-3.060,1.735)(-1.180,1.017)(-1.100,0.974)(-1.080,0.963)(-1.060,0.952)(-1.040,0.940)(-1.020,0.929)(-1,0.917)(-0.980,0.905)(-0.960,0.894)(-0.940,0.882)(-0.920,0.869)(-0.900,0.857)(-0.880,0.845)(-0.860,0.832)(-0.840,0.819)(-0.820,0.806)(-0.800,0.793)(-0.780,0.779)(-0.760,0.766)(-0.740,0.752)(-0.720,0.738)(-0.700,0.724)(-0.680,0.710)(-0.660,0.695)(-0.640,0.680)(-0.620,0.665)(-0.600,0.650)(-0.580,0.634)(-0.560,0.618)(-0.540,0.602)(-0.520,0.586)(-0.500,0.569)(-0.480,0.552)(-0.460,0.534)(-0.440,0.517)(-0.420,0.498)(-0.400,0.480)(-0.380,0.461)(-0.360,0.442)(-0.340,0.422)(-0.320,0.402)(-0.300,0.381)(-0.280,0.360)(-0.260,0.339)(-0.240,0.317)(-0.220,0.294)(-0.200,0.271)(-0.180,0.247)(-0.160,0.223)(-0.140,0.198)(-0.120,0.172)(-0.100,0.145)(-0.080,0.118)(-0.060,0.090)(-0.040,0.061)(-0.020,0.031)(0,0)(0.020,-0.032)(0.040,-0.065)(0.060,-0.099)(0.080,-0.135)(0.100,-0.171)(0.120,-0.210)(0.140,-0.249)(0.160,-0.291)(0.180,-0.334)(0.200,-0.379)(0.220,-0.426)(0.240,-0.475)(0.260,-0.526)(0.280,-0.580)(0.300,-0.637)(0.320,-0.696)(0.340,-0.759)(0.360,-0.826)(0.380,-0.896)(0.680,-2.917)(0.800,-5.344)(0.860,-8.092)(0.900,-11.744)(0.920,-14.933)(0.940,-20.245)(0.960,-30.862)(0.980,-62.700)};%
        \addplot[spectrumpos] coordinates {%
            (4.020,-0.016)(5.480,-1.389)(6.560,-3.108)(7.220,-5.018)(7.620,-6.986)(7.880,-8.981)(8.060,-10.987)(8.200,-13.156)(8.300,-15.229)(8.380,-17.362)(8.460,-20.121)(8.520,-22.790)(8.560,-24.971)(8.600,-27.586)(8.640,-30.779)(8.680,-34.769)(8.700,-37.162)(8.720,-39.895)(8.740,-43.049)};%
        \addplot[spectrumpos] coordinates {%
            (16.020,-0.016)(17.560,-1.293)(18.980,-2.732)(20.160,-4.351)(21.080,-6.141)(21.760,-8.024)(22.280,-10.038)(22.660,-12.037)(22.960,-14.115)(23.200,-16.258)(23.380,-18.270)(23.540,-20.464)(23.680,-22.813)(23.780,-24.816)(23.880,-27.171)(23.960,-29.379)(24.040,-31.950)(24.100,-34.176)(24.160,-36.718)(24.220,-39.649)(24.260,-41.865)};%
        \addplot[spectrumpos] coordinates {%
            (36.020,-0.016)(37.580,-1.272)(39.080,-2.606)(40.480,-4.071)(41.720,-5.678)(42.760,-7.396)(43.620,-9.228)(44.320,-11.150)(44.880,-13.110)(45.340,-15.134)(45.720,-17.209)(46.020,-19.202)(46.280,-21.271)(46.500,-23.347)(46.700,-25.569)(46.860,-27.639)(47,-29.716)(47.120,-31.739)(47.240,-34.033)(47.340,-36.194)(47.440,-38.629)(47.520,-40.812)};%
      \end{scope}
      \begin{scope}[tblue]
        \addplot[spectrum] coordinates {%
            (-5,2.240)(-3.040,1.758)(-1.120,1.137)(0.680,0.233)(1.140,-0.114)(1.160,-0.131)(1.180,-0.148)(1.200,-0.165)(1.220,-0.183)(1.240,-0.201)(1.260,-0.219)(1.280,-0.237)(1.300,-0.255)(1.320,-0.274)(1.340,-0.293)(1.360,-0.312)(1.380,-0.331)(1.400,-0.350)(1.420,-0.370)(1.440,-0.390)(1.460,-0.410)(1.480,-0.431)(1.500,-0.451)(1.520,-0.472)(1.540,-0.494)(1.560,-0.515)(1.580,-0.537)(1.600,-0.559)(1.620,-0.581)(1.640,-0.604)(1.660,-0.627)(1.680,-0.651)(1.700,-0.674)(1.720,-0.698)(1.740,-0.723)(1.760,-0.747)(1.780,-0.773)(1.800,-0.798)(1.820,-0.824)(1.840,-0.850)(1.860,-0.877)(1.880,-0.904)(1.900,-0.932)(1.920,-0.960)(1.940,-0.988)(1.960,-1.017)(1.980,-1.047)(2,-1.077)(2.020,-1.107)(2.040,-1.138)(2.060,-1.170)(2.080,-1.202)(2.100,-1.234)(2.120,-1.268)(2.140,-1.301)(2.160,-1.336)(2.180,-1.371)(2.880,-3.293)(3.220,-5.363)(3.400,-7.369)(3.520,-9.523)(3.600,-11.667)(3.660,-13.930)(3.700,-15.938)(3.740,-18.561)(3.780,-22.134)(3.800,-24.454)(3.820,-27.289)(3.840,-30.831)(3.860,-35.384)(3.880,-41.453)};%
        \addplot[spectrum] coordinates {%
            (4.120,43.362)(4.140,37.293)(4.160,32.740)(4.180,29.198)(4.200,26.363)(4.220,24.042)(4.260,20.469)(4.300,17.846)(4.340,15.837)(4.400,13.572)(4.480,11.427)(4.600,9.269)(4.780,7.256)(5.100,5.257)(5.740,3.343)(6.920,1.724)(8.460,0.420)(10.020,-0.842)(11.420,-2.292)(12.520,-3.993)(13.300,-5.887)(13.820,-7.821)(14.200,-9.892)(14.480,-12.049)(14.680,-14.132)(14.840,-16.302)(14.960,-18.360)(15.060,-20.469)(15.140,-22.506)(15.220,-24.957)(15.280,-27.149)(15.340,-29.737)(15.380,-31.739)(15.420,-34.015)(15.460,-36.628)(15.500,-39.656)(15.540,-43.210)};%
        \addplot[spectrum] coordinates {%
            (16.500,41.564)(16.540,38.535)(16.580,35.923)(16.620,33.646)(16.660,31.643)(16.720,29.055)(16.780,26.862)(16.860,24.410)(16.940,22.372)(17.040,20.261)(17.160,18.201)(17.320,16.028)(17.520,13.941)(17.780,11.909)(18.140,9.880)(18.640,7.930)(19.360,6.056)(20.380,4.314)(21.680,2.792)(23.180,1.443)(24.760,0.188)(26.320,-1.065)(27.820,-2.426)(29.160,-3.923)(30.300,-5.584)(31.220,-7.381)(31.940,-9.273)(32.500,-11.227)(32.940,-13.225)(33.300,-15.316)(33.580,-17.353)(33.820,-19.501)(34.020,-21.677)(34.180,-23.754)(34.320,-25.889)(34.440,-28.019)(34.540,-30.058)(34.640,-32.392)(34.720,-34.519)(34.800,-36.928)(34.860,-38.954)(34.920,-41.203)};%
        \addplot[spectrum] coordinates {%
            (37.140,40.859)(37.200,38.832)(37.280,36.423)(37.360,34.295)(37.460,31.960)(37.560,29.920)(37.680,27.789)(37.820,25.652)(37.980,23.573)(38.180,21.394)(38.400,19.406)(38.680,17.333)(39.020,15.311)(39.440,13.339)(39.980,11.375)(40.660,9.494)(41.540,7.664)(42.620,5.970)(43.900,4.417)(45.340,2.996)(46.860,1.695)(48.440,0.439)(50.020,-0.809)(51.560,-2.093)(53.040,-3.467)(54.400,-4.939)(55.620,-6.531)(56.680,-8.238)(57.580,-10.045)(58.340,-11.955)(58.960,-13.892)(59.480,-15.890)(60,-18.370)};%
      \end{scope}
    \end{axis}
  \end{tikzpicture}
  \caption{Spectrum on an interval. Dashed lines indicate where the semigroup is (strongly) positive.}
  \label{fig:graph-2-2}
\end{figure}

That alternation does not need to be the case for more general graphs as the following example shows, where the branches of eigenvalue curves come in groups.

\begin{example}
  \label{ex:nw-3-2}
  Consider the graph as shown in Figure~\ref{fig:nw-3-2}.
  \begin{figure}[ht]
    \centering
    \begin{tikzpicture}
      \node[onode] (1) at (0,0) {$v_1$};%
      \node[onode] (2) [right of=1] {$v_2$};%
      \node[inode] (3) [right of=2] {$v_3$};%
      \path (1) edge[nedge]  node {$L_{12}=1$} (2);%
      \path (2) edge[nedge]  node {$L_{23}=\sqrt{17}$} (3);%
    \end{tikzpicture}
    \caption{Simple graph with two edges and outer boundary of two nodes.}
    \label{fig:nw-3-2}
  \end{figure}
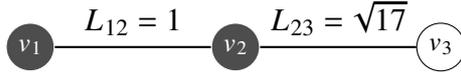
  We recall that the dark nodes represent $\Vo$, the white ones $\Vi$. Using formula \eqref{eq:DN-Vo} we can explicitly compute $D_{\lambda,\Vo}$. Since the matrix is a $2\times 2$ matrix the semigroup generated by $-D_{\lambda,\Vo}$ is either positive or does not have any positivity property; see \cite[Proposition~6.2]{daners:16:eps}. We can actually analyse this example analytically using formula~\eqref{eq:DN-Vo}. From Proposition~\ref{prop:explicit-formula-for-dtn-operator} we have
  \begin{equation*}
    D_{\lambda,V}=
    \begin{bmatrix}
      \alpha_{12} & -\beta_{12}             & 0           \\
      -\beta_{21} & \alpha_{21}+\alpha_{23} & -\beta_{23} \\
      0           & -\beta_{32}             & \alpha_{32}
    \end{bmatrix}
  \end{equation*}
  and hence by Proposition~\ref{prop:main}~\ref{item:main:formula-for-dtn-operator-with-inner-vertices} we conclude that
  \begin{equation*}
    D_{\lambda,\Vo}=
    \begin{bmatrix}
      \alpha_{12} & -\beta_{12}             \\
      -\beta_{21} & \alpha_{21}+\alpha_{23}
    \end{bmatrix}
    -\frac{1}{\alpha_{32}}
    \begin{bmatrix}
      0 \\-\beta_{23}
    \end{bmatrix}
    \begin{bmatrix}
      0 & -\beta_{32}
    \end{bmatrix}
    =
    \begin{bmatrix}
      \alpha_{12} & -\beta_{12}                                                      \\
      -\beta_{21} & \alpha_{21}+\alpha_{23}-\frac{\beta_{23}\beta_{32}}{\alpha_{32}}
    \end{bmatrix}
  \end{equation*}
  Now the semigroup generated by $-D_{\lambda,\Vo}$ is positive if and only if $\lambda\not\in\sigma(-\Delta_{\Vo})$ and
  \begin{equation*}
    \beta_{12}=\dfrac{\sqrt{\lambda}}{\sin\left(\sqrt{\lambda}L_{12}\right)}>0.
  \end{equation*}
  This is the case precisely when either $\lambda\leq 0$ or $\lambda\not\in\sigma(-\Delta_{\Vo})$ and
  \begin{equation*}
    \left(\frac{k\pi}{L_{12}}\right)^2<\lambda<\left(\frac{(k+1)\pi}{L_{12}}\right)^2
  \end{equation*}
  for some $k\geq 0$ even. We plot the spectrum $\sigma(D_{\lambda,\Vo})$ in Figure~\ref{fig:nw-3-2-spectrum}. The vertical lines are asymptotes to the curves of eigenvalues and correspond to the spectrum $\sigma(-\Delta_{\Vo})$. The semigroup is (strongly) positive where the dominant eigenvalue has a positive eigenfunction. We again use dashed lines to indicate where this happens. We observe that the semigroup stays positive across several eigenvalues beyond the first eigenvalue $\lambda_1$ of $-\Delta_{\Vo}$. We also observe that the semigroup can be stable and positive for some intervals of $\lambda>\lambda_1$.

  \begin{figure}[H]
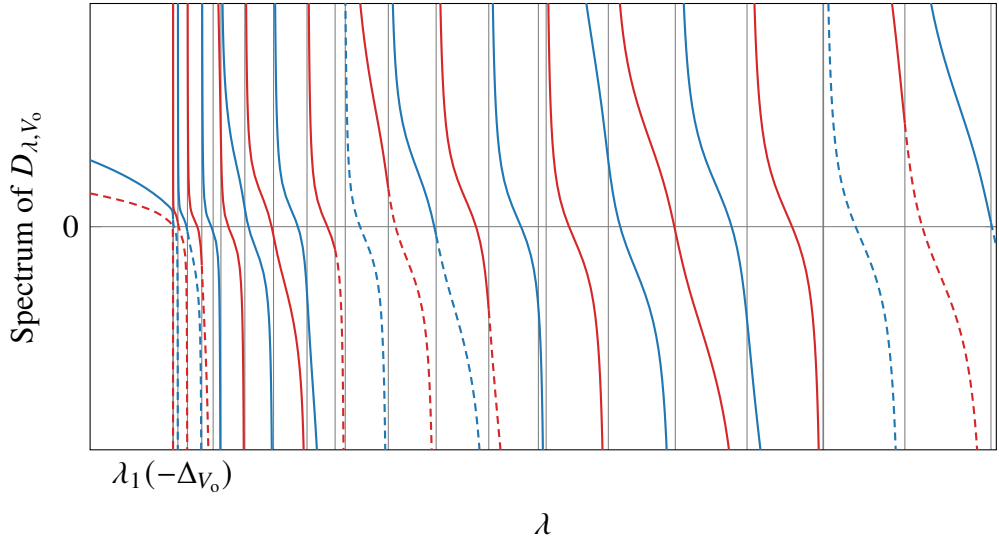

    \centering

    \caption{Spectrum of $D_{\lambda,\Vo}$ corresponding to the graph in Figure~\ref{fig:nw-3-2}.}
    \label{fig:nw-3-2-spectrum}
  \end{figure}
\end{example}

We next present an example to show that the assertion of Theorem~\ref{thm:main-result}\ref{thm:main:evpos} is not valid if $G_{\outerVertex}$ does not contain a cycle.

\begin{example}
  \label{ex:evpos-counterexample}
  Consider the graph with nodes $V=\{v_k\colon k=1,\dots,5\}$ and edges $(1,2)$, $(2,3)$, $(1,4)$, $(2,4)$ and $(4,5)$. We let $\Vo=\{v_1,v_2,v_3\}$. The graph $G$ is shown in Figure~\ref{fig:nw-3-3} on the left, and the corresponding reduced graph $G_{\outerVertex}$ on the right.
  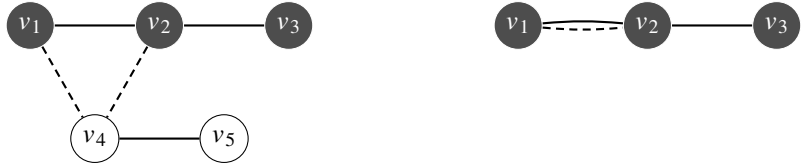
\begin{figure}[ht]
    \centering
    \tikzset{declare function={dh=1.7;dv=1.5;}}
    \begin{tikzpicture}[baseline=(1.center)]
      \node[onode] (1) at (0,0) {$v_1$};%
      \node[onode] (2) at (dh,0) {$v_2$};%
      \node[onode] (3) at (2*dh,0) {$v_3$};%
      \node[inode] (4) at (0.5*dh,-dv) {$v_4$};%
      \node[inode] (5) at (1.5*dh,-dv) {$v_5$};%
      \path (1) edge[iedge] (2);%
      \path (2) edge[iedge] (3);%
      \path (1) edge[cedge] (4);%
      \path (2) edge[cedge] (4);%
      \path (4) edge[iedge] (5);%
    \end{tikzpicture}
    \hfil
    \begin{tikzpicture}[baseline=(1.center)]
      \node[onode] (1) at (0,0) {$v_1$};%
      \node[onode] (2) at (dh,0) {$v_2$};%
      \node[onode] (3) at (2*dh,0) {$v_3$};%
      \path (1) edge[iedge,out=5,in=175] (2);%
      \path (1) edge[cedge,out=-5,in=-175] (2);%
      \path (2) edge[iedge] (3);%
    \end{tikzpicture}
    \caption{Graph $G$ on the left, and the corresponding reduced graph $G_{\outerVertex}$ on the right.}
    \label{fig:nw-3-3}
  \end{figure}

  Clearly, $G_{\outerVertex}$ does not have any cycles; removing any of the edges of $G_{\outerVertex}$ leaves the graph disconnected. We know from Proposition~\ref{prop:properties-for-dtn-operator-with-inner-vertices} that $[D_{\lambda,\Vo}]_{13}=[D_{\lambda,\Vo}]_{31}=0$. Denote entries that are positive by $+$ and entries that are arbitrary by $\ast$. It follows from \cite[Theorem~6.4]{berman:10:spe} (or, more generally, \cite[Theorem~1]{johnson:22:tsp}) that the semigroup is eventually strongly positive if and only if $-D_{\lambda,\Vo}$ has the form
  \begin{equation*}
    -D_{\lambda,\Vo}=
    \begin{bmatrix}
      * & + & 0 \\
      + & * & + \\
      0 & + & *
    \end{bmatrix}.
  \end{equation*}
  Hence, either the semigroup $(e^{-tD_{\lambda,\Vo}})_{t>0}$ is positive or not eventually positive. We also note that in this example the original graph $G$ \emph{does} contain a cycle. 
\end{example}
\begin{figure}[ht]
  \centering
  \begin{tikzpicture}
    \begin{axis}[%
        xmin=-5,%
        xmax=50,%
        ymin=-20,%
        ymax=20,%
        xtick={0.172,3.222,8.375,15.038,22.565,38.022,45.058,0.912,4.938,9.872,17.505,28.928,39.482,2.368,5.645,10.782,19.742,30.845,44.415},%
        ytick={0},%
        extra x ticks={0.172},
        extra x tick labels={$\lambda_1(-\Delta_{\Vo})$},
        xlabel={$\lambda$},%
        ylabel={Spectrum of $D_{\lambda,\Vo}$},%
      ]%
      \begin{scope}[tred]
        \addplot[spectrum] coordinates {%
            (0.170,178.196)(0.180,6.828)(0.190,3.707)(0.210,2.837)(0.220,2.818)(0.230,2.803)(0.240,2.790)(0.250,2.777)(0.260,2.765)(0.270,2.753)(0.690,2.216)(0.700,2.202)(0.710,2.187)(0.720,2.173)(0.730,2.158)(0.740,2.143)(0.750,2.128)(0.760,2.113)(0.770,2.098)(0.780,2.083)(0.790,2.067)(0.800,2.052)(0.810,2.036)(0.820,2.020)(0.830,2.004)(0.840,1.987)(0.850,1.971)(0.860,1.954)(0.870,1.937)(0.880,1.919)(0.890,1.902)(0.900,1.884)(0.910,1.866)(0.920,1.848)(0.930,1.829)(0.940,1.810)(0.950,1.790)(0.960,1.771)(0.970,1.750)(0.980,1.730)(0.990,1.708)(1,1.687)(1.010,1.664)(1.020,1.642)(1.030,1.618)(1.040,1.594)(1.050,1.569)(1.060,1.544)(1.070,1.517)(1.080,1.490)(1.090,1.462)(1.100,1.434)(1.110,1.404)(1.120,1.374)(1.430,0.422)(1.770,-0.539)(2.030,-1.527)(2.280,-2.531)(2.360,-2.850)};%
        \addplot[spectrum] coordinates {%
            (3.290,22.552)(3.300,19.902)(3.310,17.842)(3.320,16.195)(3.330,14.849)(3.340,13.729)(3.360,11.973)(3.380,10.659)(3.400,9.639)(3.430,8.477)(3.470,7.361)(3.530,6.223)(3.620,5.140)(3.760,4.128)(4,3.136)(4.390,2.205)(4.930,1.359)(5.550,0.566)(5.770,0.276)(5.780,0.262)(5.790,0.248)(5.800,0.234)(5.810,0.220)(5.820,0.205)(5.830,0.191)(5.840,0.176)(5.850,0.161)(5.860,0.146)(5.870,0.131)(5.880,0.115)(5.890,0.100)(5.900,0.084)(5.910,0.067)(5.920,0.051)(5.930,0.034)(5.940,0.017)(5.950,-0)(5.960,-0.018)(5.970,-0.036)(5.980,-0.055)(5.990,-0.074)(6,-0.094)(6.010,-0.114)(6.020,-0.135)(6.030,-0.157)(6.040,-0.180)(6.050,-0.204)(6.060,-0.228)(6.070,-0.254)(6.080,-0.281)(6.090,-0.310)(6.100,-0.341)(6.110,-0.373)(6.120,-0.407)(6.130,-0.444)(6.140,-0.483)(6.150,-0.525)(6.160,-0.570)(6.170,-0.619)(6.180,-0.670)(6.320,-1.694)(6.440,-2.711)(6.570,-3.767)(6.700,-4.778)(6.830,-5.771)(6.960,-6.774)(7.090,-7.808)(7.210,-8.812)(7.330,-9.882)(7.440,-10.944)(7.540,-11.998)(7.630,-13.041)(7.710,-14.069)(7.780,-15.070)(7.850,-16.203)(7.910,-17.321)(7.960,-18.400)(8,-19.400)(8.040,-20.573)};%
        \addplot[spectrum] coordinates {%
            (8.530,20.467)(8.540,19.137)(8.550,17.963)(8.560,16.920)(8.580,15.150)(8.600,13.706)(8.620,12.509)(8.640,11.501)(8.670,10.259)(8.700,9.259)(8.740,8.196)(8.790,7.174)(8.860,6.119)(8.960,5.073)(9.110,4.047)(9.340,3.056)(9.680,2.099)(10.120,1.187)(10.620,0.315)(11.150,-0.538)(11.690,-1.396)(12.210,-2.252)(12.710,-3.135)(13.170,-4.029)(13.590,-4.954)(13.950,-5.907)(14.200,-6.893)(14.320,-7.907)(14.390,-9.087)(14.440,-10.353)(14.480,-11.647)(14.510,-12.793)(14.540,-14.105)(14.570,-15.607)(14.590,-16.728)(14.610,-17.959)(14.630,-19.314)(14.650,-20.813)};%
        \addplot[spectrum] coordinates {%
            (15.860,20.454)(15.920,19.367)(15.990,18.245)(16.070,17.116)(16.150,16.115)(16.240,15.108)(16.340,14.103)(16.450,13.101)(16.570,12.094)(16.700,11.069)(16.830,10.075)(16.960,9.074)(17.090,8.028)(17.210,6.987)(17.320,5.947)(17.420,4.934)(17.520,3.900)(17.630,2.822)(17.750,1.811)(17.910,0.767)(18.120,-0.217)(18.410,-1.188)(18.770,-2.133)(19.160,-3.060)(19.540,-3.985)(19.890,-4.927)(20.200,-5.889)(20.470,-6.880)(20.700,-7.889)(20.890,-8.886)(21.050,-9.887)(21.190,-10.926)(21.310,-11.981)(21.410,-13.012)(21.500,-14.094)(21.580,-15.211)(21.650,-16.342)(21.710,-17.453)(21.760,-18.502)(21.810,-19.686)(21.850,-20.751)};%
        \addplot[spectrum] coordinates {%
            (23.550,20.528)(23.650,19.460)(23.770,18.386)(23.910,17.330)(24.070,16.305)(24.250,15.313)(24.460,14.307)(24.700,13.302)(24.970,12.304)(25.270,11.317)(25.600,10.343)(25.960,9.380)(26.340,8.452)(26.750,7.530)(27.190,6.612)(27.650,5.715)(28.130,4.832)(28.630,3.950)(29.120,3.064)(29.260,2.775)(29.270,2.752)(29.280,2.729)(29.290,2.706)(29.300,2.683)(29.310,2.659)(29.320,2.635)(29.330,2.610)(29.340,2.585)(29.350,2.560)(29.360,2.533)(29.370,2.507)(29.380,2.479)(29.390,2.451)(29.400,2.422)(29.410,2.392)(29.420,2.361)(29.430,2.330)(29.440,2.297)(29.450,2.263)(29.460,2.228)(29.470,2.191)(29.480,2.153)(29.490,2.114)(29.500,2.072)(29.510,2.029)(29.520,1.983)(29.530,1.935)(29.670,0.869)(29.750,-0.148)(29.820,-1.162)(29.890,-2.180)(29.970,-3.284)(30.050,-4.307)(30.140,-5.361)(30.240,-6.424)(30.350,-7.478)(30.470,-8.516)(30.600,-9.533)(30.740,-10.531)(30.900,-11.579)(31.070,-12.613)(31.250,-13.645)(31.430,-14.635)(31.620,-15.652)(31.810,-16.655)(32,-17.656)(32.190,-18.664)(32.380,-19.687)(32.560,-20.678)};%
        \addplot[spectrum] coordinates {%
            (38.160,23.497)(38.170,13.844)(38.180,5.957)(38.190,3.228)(38.200,3.081)(38.210,3.035)(38.220,3.008)(38.230,2.987)(38.240,2.970)(38.250,2.954)(38.260,2.940)(38.270,2.926)(38.280,2.913)(38.290,2.900)(38.300,2.887)(38.310,2.874)(38.320,2.862)(38.330,2.850)(38.340,2.838)(38.350,2.826)(39.020,2.072)(39.710,1.333)(40.400,0.607)(41.090,-0.124)(41.770,-0.865)(42.430,-1.621)(43.070,-2.410)(43.670,-3.224)(44.220,-4.078)(44.320,-4.257)(44.330,-4.276)(44.340,-4.295)(44.350,-4.315)(44.360,-4.334)(44.370,-4.354)(44.380,-4.374)(44.390,-4.395)(44.400,-4.416)(44.410,-4.437)(44.420,-4.459)(44.430,-4.481)(44.440,-4.504)(44.450,-4.527)(44.460,-4.551)(44.470,-4.576)(44.480,-4.602)(44.490,-4.629)(44.500,-4.658)(44.510,-4.688)(44.520,-4.719)(44.530,-4.753)(44.540,-4.790)(44.550,-4.830)(44.560,-4.874)(44.570,-4.924)(44.580,-4.981)(44.590,-5.048)(44.600,-5.127)(44.660,-6.434)(44.680,-7.613)(44.700,-9.252)(44.720,-11.271)(44.730,-12.413)(44.740,-13.645)(44.750,-14.972)(44.760,-16.401)(44.770,-17.943)(44.780,-19.606)(44.790,-21.407)};%
      \end{scope}
      \begin{scope}[tred]
        \addplot[spectrumpos] coordinates {%
            (-5,2.236)(-4.020,2.005)(-3.050,1.746)(-2.090,1.443)(-1.160,1.065)(-1.050,1.011)(-1.040,1.005)(-1.030,1)(-1.020,0.995)(-1.010,0.990)(-1,0.985)(-0.990,0.979)(-0.980,0.974)(-0.970,0.969)(-0.960,0.963)(-0.950,0.958)(-0.940,0.952)(-0.930,0.947)(-0.920,0.941)(-0.910,0.936)(-0.900,0.930)(-0.890,0.924)(-0.880,0.919)(-0.870,0.913)(-0.860,0.907)(-0.850,0.901)(-0.840,0.896)(-0.830,0.890)(-0.820,0.884)(-0.810,0.878)(-0.800,0.872)(-0.790,0.866)(-0.780,0.860)(-0.770,0.854)(-0.760,0.847)(-0.750,0.841)(-0.740,0.835)(-0.730,0.829)(-0.720,0.822)(-0.710,0.816)(-0.700,0.809)(-0.690,0.803)(-0.680,0.796)(-0.670,0.789)(-0.660,0.783)(-0.650,0.776)(-0.640,0.769)(-0.630,0.762)(-0.620,0.755)(-0.610,0.748)(-0.600,0.741)(-0.590,0.734)(-0.580,0.727)(-0.570,0.719)(-0.560,0.712)(-0.550,0.705)(-0.540,0.697)(-0.530,0.689)(-0.520,0.682)(-0.510,0.674)(-0.500,0.666)(-0.490,0.658)(-0.480,0.650)(-0.470,0.642)(-0.460,0.634)(-0.450,0.625)(-0.440,0.617)(-0.430,0.608)(-0.420,0.599)(-0.410,0.591)(-0.400,0.582)(-0.390,0.573)(-0.380,0.564)(-0.370,0.554)(-0.360,0.545)(-0.350,0.535)(-0.340,0.526)(-0.330,0.516)(-0.320,0.506)(-0.310,0.495)(-0.300,0.485)(-0.290,0.475)(-0.280,0.464)(-0.270,0.453)(-0.260,0.442)(-0.250,0.430)(-0.240,0.419)(-0.230,0.407)(-0.220,0.395)(-0.210,0.382)(-0.200,0.370)(-0.190,0.357)(-0.180,0.344)(-0.170,0.330)(-0.160,0.316)(-0.150,0.302)(-0.140,0.287)(-0.130,0.272)(-0.120,0.256)(-0.110,0.240)(-0.100,0.223)(-0.090,0.205)(-0.080,0.187)(-0.070,0.168)(-0.060,0.148)(-0.050,0.127)(-0.040,0.104)(-0.030,0.081)(-0.020,0.056)(-0.010,0.029)(0,-0)(0.010,-0.031)(0.020,-0.066)(0.030,-0.104)(0.040,-0.146)(0.050,-0.194)(0.060,-0.250)(0.070,-0.315)(0.080,-0.392)(0.090,-0.488)(0.140,-1.889)(0.150,-3.027)(0.160,-20)};%
        \addplot[spectrumpos] coordinates {%
            (2.370,-2.890)(2.600,-3.885)(2.780,-4.892)(2.900,-5.924)(2.970,-6.943)(3.020,-8.169)(3.050,-9.317)(3.070,-10.381)(3.090,-11.818)(3.110,-13.834)(3.120,-15.169)(3.130,-16.819)(3.140,-18.902)(3.150,-21.606)};%
      \end{scope}
      \begin{scope}[tblue]
        \addplot[spectrum] coordinates {%
            (-5,4.425)(-4.100,3.988)(-3.220,3.508)(-2.370,2.979)(-1.560,2.392)(-0.800,1.738)(-0.220,1.150)(-0.210,1.139)(-0.200,1.127)(-0.190,1.116)(-0.180,1.104)(-0.170,1.093)(-0.160,1.081)(-0.150,1.069)(-0.140,1.057)(-0.130,1.045)(-0.120,1.033)(-0.110,1.021)(-0.100,1.008)(-0.090,0.996)(-0.080,0.983)(-0.070,0.970)(-0.060,0.957)(-0.050,0.943)(-0.040,0.930)(-0.030,0.916)(-0.020,0.901)(-0.010,0.887)(0,0.872)(0.010,0.856)(0.020,0.840)(0.030,0.824)(0.040,0.807)(0.050,0.789)(0.060,0.770)(0.070,0.751)(0.080,0.730)(0.090,0.708)(0.100,0.685)(0.110,0.661)(0.120,0.634)(0.130,0.606)(0.140,0.576)(0.150,0.544)(0.160,0.510)(0.170,0.474)(0.190,0.398)};%
        \addplot[spectrum] coordinates {%
            (0.930,24.557)(0.940,16.512)(0.950,12.460)(0.960,10.018)(0.970,8.386)(0.980,7.217)(1,5.657)(1.030,4.292)(1.070,3.275)(1.160,2.222)(1.170,2.155)(1.180,2.093)(1.190,2.037)(1.200,1.984)(1.210,1.936)(1.220,1.891)(1.230,1.849)(1.240,1.810)(1.250,1.773)(1.260,1.738)(1.270,1.704)(1.280,1.672)(1.290,1.642)(1.300,1.612)(1.310,1.584)(1.320,1.556)(1.330,1.530)(1.340,1.503)(1.350,1.478)(1.360,1.453)(1.370,1.429)(1.380,1.405)(1.390,1.381)(1.830,0.479)(2.320,-0.414)(2.460,-0.709)(2.470,-0.732)(2.480,-0.755)(2.490,-0.779)(2.500,-0.803)(2.510,-0.828)(2.520,-0.852)(2.530,-0.878)(2.540,-0.903)(2.550,-0.929)(2.560,-0.956)(2.570,-0.983)(2.580,-1.011)(2.590,-1.039)(2.600,-1.067)(2.610,-1.097)(2.620,-1.126)(2.860,-2.140)(3,-3.220)(3.100,-4.229)(3.210,-5.280)};%
        \addplot[spectrum] coordinates {%
            (5.540,21.094)(5.570,19.729)(5.600,18.411)(5.630,17.116)(5.660,15.822)(5.690,14.511)(5.720,13.176)(5.750,11.827)(5.780,10.493)(5.810,9.213)(5.840,8.023)(5.870,6.939)(5.910,5.667)(5.950,4.577)(6,3.432)(6.060,2.317)(6.130,1.311)(6.230,0.399)(6.240,0.339)(6.250,0.284)(6.260,0.234)(6.270,0.188)(6.280,0.144)(6.290,0.104)(6.300,0.067)(6.310,0.032)(6.320,-0.001)(6.330,-0.032)(6.340,-0.061)(6.350,-0.089)(6.360,-0.116)(6.370,-0.141)(6.380,-0.166)(6.390,-0.189)(6.400,-0.212)(6.410,-0.235)(6.420,-0.257)(6.430,-0.278)(6.440,-0.299)(6.450,-0.319)(6.460,-0.339)(6.470,-0.359)(6.480,-0.378)(6.490,-0.397)(6.500,-0.416)(6.980,-1.298)(7.360,-2.228)(7.620,-3.194)(7.800,-4.216)(7.920,-5.218)(8.010,-6.267)(8.080,-7.361)(8.140,-8.576)(8.190,-9.838)(8.230,-11.028)(8.270,-12.365)(8.300,-13.437)(8.330,-14.537)(8.360,-15.638)(8.390,-16.718)(8.420,-17.767)(8.450,-18.785)(8.490,-20.102)};%
        \addplot[spectrum] coordinates {%
            (11.130,20.099)(11.160,18.687)(11.190,17.455)(11.220,16.373)(11.260,15.121)(11.300,14.046)(11.350,12.900)(11.410,11.748)(11.480,10.637)(11.560,9.595)(11.660,8.541)(11.780,7.531)(11.930,6.538)(12.130,5.516)(12.380,4.541)(12.700,3.577)(13.080,2.641)(13.470,1.710)(13.770,0.722)(13.950,-0.311)(14.070,-1.386)(14.160,-2.451)(14.240,-3.566)(14.310,-4.588)(14.390,-5.590)(14.530,-6.632)(14.760,-7.624)(15.020,-8.607)(15.270,-9.601)(15.500,-10.610)(15.710,-11.642)(15.890,-12.636)(16.050,-13.630)(16.200,-14.683)(16.330,-15.718)(16.450,-16.806)(16.550,-17.840)(16.640,-18.901)(16.720,-19.982)(16.790,-21.068)};%
        \addplot[spectrum] coordinates {%
            (17.860,21.186)(17.890,19.858)(17.920,18.739)(17.960,17.497)(18,16.471)(18.050,15.414)(18.110,14.386)(18.190,13.300)(18.290,12.247)(18.420,11.198)(18.580,10.209)(18.790,9.211)(19.060,8.215)(19.390,7.249)(19.780,6.317)(20.230,5.413)(20.730,4.547)(21.280,3.704)(21.870,2.884)(22.480,2.088)(22.560,1.986)};%
        \addplot[spectrum] coordinates {%
            (24.170,-0.377)(24.520,-1.316)(24.820,-2.298)(25.110,-3.273)(25.400,-4.238)(25.690,-5.199)(25.980,-6.175)(26.260,-7.153)(26.520,-8.120)(26.770,-9.133)(26.990,-10.125)(27.190,-11.146)(27.360,-12.140)(27.510,-13.152)(27.640,-14.170)(27.750,-15.167)(27.850,-16.218)(27.940,-17.316)(28.020,-18.450)(28.090,-19.599)(28.150,-20.732)};%
        \addplot[spectrum] coordinates {%
            (29.280,21.144)(29.300,19.421)(29.320,17.866)(29.340,16.456)(29.360,15.171)(29.380,13.996)(29.400,12.916)(29.430,11.453)(29.460,10.153)(29.490,8.995)(29.520,7.962)(29.560,6.762)(29.600,5.754)(29.650,4.754)(29.730,3.701)(29.820,3.063)(29.830,3.013)(29.840,2.965)(29.850,2.921)(29.860,2.878)(29.870,2.837)(29.880,2.799)(29.890,2.762)(29.900,2.726)(29.910,2.692)(29.920,2.659)(29.930,2.628)(29.940,2.597)(29.950,2.568)(29.960,2.539)(29.970,2.512)(29.980,2.485)(29.990,2.458)(30,2.433)(30.010,2.408)(30.020,2.383)(30.030,2.360)(30.040,2.336)(30.050,2.313)(30.060,2.291)(30.070,2.269)(30.080,2.247)(30.090,2.226)(30.100,2.205)(30.110,2.184)(30.120,2.164)(30.130,2.144)(30.140,2.124)(30.150,2.105)(30.160,2.085)(30.170,2.066)(30.180,2.047)(30.190,2.029)(30.730,1.175)(31.330,0.358)(31.950,-0.444)(32.570,-1.235)(33.190,-2.037)(33.790,-2.842)(34.370,-3.664)(34.930,-4.513)(35.460,-5.384)(35.950,-6.261)(36.410,-7.164)(36.830,-8.071)(37.220,-9.005)(37.570,-9.945)(37.880,-10.933)(38.010,-11.491)};%
        \addplot[spectrum] coordinates {%
            (38.190,-17.123)(38.200,-20.305)};%
        \addplot[spectrum] coordinates {%
            (43.910,20.482)(43.970,19.377)(44.030,18.225)(44.080,17.219)(44.130,16.164)(44.180,15.047)(44.230,13.855)(44.270,12.837)(44.310,11.750)(44.350,10.580)(44.390,9.312)(44.420,8.285)(44.450,7.181)(44.480,5.989)(44.510,4.695)(44.540,3.284)(44.560,2.274)(44.580,1.204)(44.600,0.083)(44.620,-1.068)(44.640,-2.182)(44.670,-3.470)(44.720,-4.322)(44.730,-4.401)(44.740,-4.467)(44.750,-4.524)(44.760,-4.574)(44.770,-4.619)(44.780,-4.660)(44.790,-4.697)(44.800,-4.732)(44.810,-4.765)(44.820,-4.796)(44.830,-4.826)(44.840,-4.854)(44.850,-4.881)(44.860,-4.908)(44.870,-4.934)(44.880,-4.959)(44.890,-4.984)(44.900,-5.008)(44.910,-5.031)(45.340,-5.948)(45.750,-6.865)(46.130,-7.819)(46.470,-8.789)(46.770,-9.761)(47.040,-10.752)(47.280,-11.747)(47.490,-12.726)(47.680,-13.717)(47.850,-14.704)(48.010,-15.737)(48.150,-16.737)(48.280,-17.762)(48.400,-18.803)(48.510,-19.850)(48.610,-20.891)};%
      \end{scope}
      \begin{scope}[tblue]
        \addplot[spectrumpos] coordinates {%
            (0.200,0.358)(0.490,-0.623)(0.670,-1.682)(0.750,-2.781)(0.790,-3.844)(0.820,-5.245)(0.840,-6.838)(0.850,-8.031)(0.860,-9.701)(0.870,-12.209)(0.880,-16.401)(0.890,-24.841)};%
        \addplot[spectrumpos] coordinates {%
            (3.220,-5.368)(3.340,-6.366)(3.470,-7.429)(3.590,-8.487)(3.700,-9.577)(3.790,-10.590)(3.870,-11.612)(3.940,-12.624)(4.010,-13.773)(4.070,-14.895)(4.120,-15.946)(4.170,-17.127)(4.210,-18.183)(4.250,-19.356)(4.290,-20.669)};%
        \addplot[spectrumpos] coordinates {%
            (22.570,1.973)(23.190,1.174)(23.770,0.357)(24.160,-0.355)};%
        \addplot[spectrumpos] coordinates {%
            (38.020,-11.545)(38.130,-12.597)(38.170,-14.180)(38.190,-17.123)};%
      \end{scope}
      \begin{scope}[torange]
        \addplot[spectrum] coordinates {%
            (-5,6.875)(-4.150,6.341)(-3.310,5.779)(-2.490,5.192)(-1.690,4.577)(-0.920,3.930)(-0.190,3.244)(0.090,2.949)(0.100,2.938)(0.110,2.926)(0.120,2.915)(0.130,2.903)(0.140,2.892)(0.150,2.880)(0.160,2.867)(0.170,2.854)(0.180,2.837)(0.190,2.807)(0.220,1.751)(0.250,1.289)(0.260,1.205)(0.270,1.136)(0.280,1.080)(0.290,1.032)(0.300,0.991)(0.310,0.955)(0.320,0.923)(0.330,0.894)(0.340,0.867)(0.350,0.843)(0.360,0.820)(0.370,0.798)(0.380,0.778)(0.390,0.759)(0.400,0.740)(0.410,0.722)(0.420,0.705)(0.430,0.688)(0.440,0.672)(0.450,0.656)(0.460,0.641)(0.470,0.625)(0.480,0.610)(0.490,0.596)(0.500,0.581)(0.510,0.567)(0.520,0.552)(0.770,0.197)(0.780,0.182)(0.790,0.166)(0.800,0.151)(0.810,0.136)(0.820,0.120)(0.830,0.104)(0.840,0.088)(0.850,0.072)(0.860,0.056)(0.870,0.039)(0.880,0.023)(0.890,0.006)(0.900,-0.011)(0.910,-0.028)(0.920,-0.045)(0.930,-0.062)(0.940,-0.080)(0.950,-0.098)(0.960,-0.116)(0.970,-0.134)(1.280,-0.784)};%
        \addplot[spectrum] coordinates {%
            (1.690,-2.001)(1.880,-3.015)(1.990,-4.075)(2.060,-5.180)(2.110,-6.361)(2.150,-7.713)(2.180,-9.116)(2.200,-10.336)(2.220,-11.894)(2.240,-13.951)(2.250,-15.246)(2.260,-16.788)(2.270,-18.654)(2.280,-20.958)};%
        \addplot[spectrum] coordinates {%
            (2.460,20.421)(2.470,18.537)(2.480,16.978)(2.490,15.668)(2.500,14.549)(2.520,12.740)(2.540,11.337)(2.560,10.216)(2.590,8.897)(2.620,7.874)(2.660,6.815)(2.710,5.805)(2.780,4.744)(2.870,3.717)(2.980,2.721)(3.110,1.725)(3.270,0.713)(3.490,-0.274)(3.790,-1.239)(4.110,-2.210)(4.390,-3.203)(4.610,-4.193)(4.780,-5.186)(4.920,-6.257)(5.020,-7.253)(5.100,-8.269)(5.170,-9.404)(5.220,-10.421)(5.270,-11.692)(5.310,-12.968)(5.340,-14.138)(5.370,-15.559)(5.390,-16.692)(5.410,-18.018)(5.430,-19.594)(5.450,-21.498)};%
        \addplot[spectrum] coordinates {%
            (5.880,20.617)(5.900,19.614)(5.930,18.373)(5.960,17.357)(6,16.242)(6.050,15.109)(6.110,14.002)(6.180,12.941)(6.260,11.934)(6.360,10.888)(6.480,9.847)(6.620,8.841)(6.790,7.829)(6.990,6.840)(7.230,5.848)(7.510,4.860)(7.820,3.880)(8.130,2.917)(8.410,1.955)(8.660,0.953)(8.890,-0.058)(9.110,-1.050)(9.320,-2.044)(9.510,-3.052)(9.670,-4.058)(9.800,-5.052)(9.910,-6.081)(10,-7.110)(10.080,-8.226)(10.140,-9.233)(10.200,-10.438)(10.250,-11.641)(10.290,-12.775)(10.330,-14.106)(10.360,-15.267)(10.390,-16.604)(10.410,-17.615)(10.430,-18.741)(10.450,-20.002)};%
        \addplot[spectrum] coordinates {%
            (11.930,20.073)(12.020,19.064)(12.120,18.016)(12.230,16.939)(12.340,15.927)(12.460,14.885)(12.590,13.815)(12.720,12.791)(12.860,11.727)(13,10.688)(13.140,9.661)(13.280,8.633)(13.420,7.595)(13.560,6.539)(13.700,5.473)(13.840,4.429)(13.990,3.428)(14.190,2.441)(14.520,1.492)(15.010,0.616)(15.570,-0.233)(16.090,-1.090)(16.520,-2.020)(16.820,-2.997)(17.020,-4.023)(17.150,-5.028)(17.250,-6.122)(17.320,-7.129)(17.380,-8.189)(17.430,-9.223)(17.480,-10.387)(17.520,-11.396)(17.560,-12.453)(17.600,-13.538)(17.640,-14.631)(17.680,-15.720)(17.720,-16.795)(17.760,-17.852)(17.800,-18.893)(17.840,-19.919)(17.880,-20.934)};%
        \addplot[spectrum] coordinates {%
            (22.380,20.199)(22.470,19.101)(22.560,17.986)(22.640,16.968)(22.720,15.918)(22.800,14.832)(22.880,13.715)(22.960,12.580)(23.040,11.449)(23.120,10.345)(23.200,9.291)(23.290,8.182)(23.380,7.164)(23.480,6.142)(23.590,5.139)(23.720,4.102)(23.870,3.077)(24.050,2.053)(24.270,1.076)(24.590,0.120)(25.050,-0.769)(25.570,-1.643)(26.060,-2.515)(26.510,-3.422)(26.910,-4.358)(27.260,-5.313)(27.570,-6.292)(27.840,-7.268)(28.080,-8.249)(28.300,-9.255)(28.500,-10.273)(28.680,-11.294)(28.840,-12.307)(28.920,-12.861)};%
        \addplot[spectrum] coordinates {%
            (32.460,20.217)(32.560,19.137)(32.670,18.076)(32.790,17.047)(32.930,15.984)(33.080,14.980)(33.250,13.980)(33.450,12.952)(33.670,11.968)(33.920,10.998)(34.210,10.024)(34.540,9.068)(34.920,8.118)(35.340,7.210)(35.810,6.324)(36.330,5.464)(36.900,4.628)(37.510,3.820)(38.040,3.166)(38.050,3.154)(38.060,3.141)(38.070,3.128)(38.080,3.115)(38.090,3.102)(38.100,3.088)(38.110,3.074)(38.120,3.059)(38.130,3.043)(38.140,3.025)(38.150,3.003)(38.160,2.972)(38.170,2.915)(38.190,-1.007)(38.200,-5.167)(38.210,-7.557)(38.220,-8.846)(38.240,-10.068)(38.290,-11.118)(38.440,-12.134)(38.670,-13.114)(38.910,-14.100)(39.140,-15.096)(39.360,-16.120)(39.470,-16.663)};%
        \addplot[spectrum] coordinates {%
            (46.400,20.969)(46.520,19.929)(46.650,18.916)(46.800,17.866)(46.960,16.857)(47.140,15.835)(47.340,14.811)(47.560,13.795)(47.800,12.792)(48.060,11.809)(48.340,10.848)(48.650,9.880)(48.990,8.914)(49.360,7.956)(50, 6.480)};%
      \end{scope}
      \begin{scope}[torange]
        \addplot[spectrumpos] coordinates {%
            (1.400,-1.085)(1.630,-1.776)(1.680,-1.962)};%
        \addplot[spectrumpos] coordinates {%
            (28.930,-12.932)(29.060,-13.925)(29.180,-14.960)(29.290,-16.040)(29.380,-17.043)(29.460,-18.047)(29.540,-19.182)(29.610,-20.305)};%
        \addplot[spectrumpos] coordinates {%
            (39.480,-16.713)(39.670,-17.712)(39.850,-18.734)(40.020,-19.774)(40.180,-20.831)};%
      \end{scope}
    \end{axis}
  \end{tikzpicture}
  \caption{Spectrum of the graph in Figure~\ref{fig:nw-3-3}. Dashed parts correspond to a positive Dirichlet-to-Neumann semigroup.}
  \label{fig:nw-3-3-spectrum}
\end{figure}
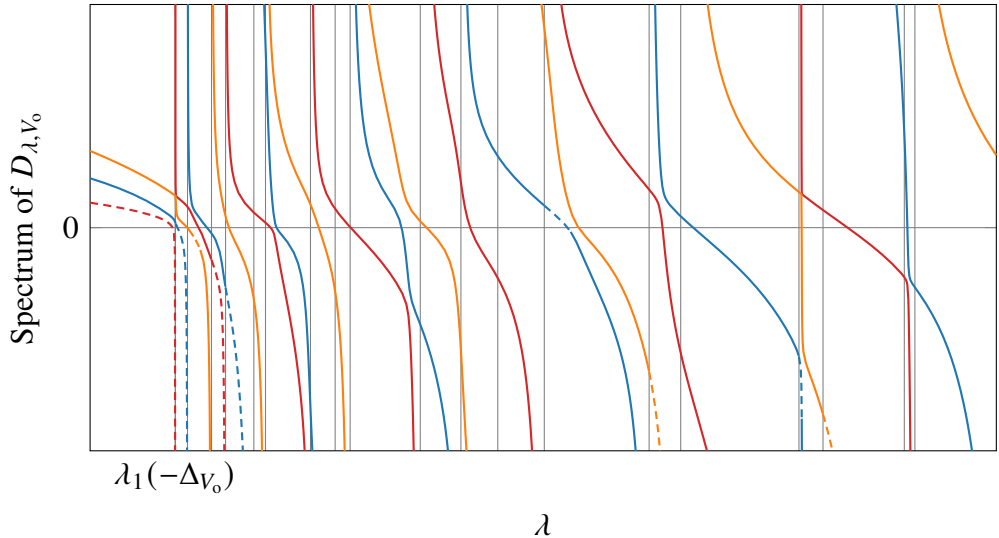

\begin{example}
  \label{ex:nw-4-3}
  We can also get an idea of the spectrum of the graph in Figure~\ref{fig:nw-4-3} with $L_{14}=1$, $L_{24}=\sqrt{3}$, $L_{34}=\sqrt{5}$ and $L_{23}=\sqrt{7}$. The reduced graph is a cycle, so satisfies the assumptions of Theorem~\ref{thm:main-result}\ref{thm:main:evpos}. The spectrum is plotted in Figure~\ref{fig:nw-4-3-spectrum}. The dashed lines indicate where the semigroup generated by $-D_{\lambda,\Vo}$ is strongly positive, the dotted ones where it is eventually strongly positive but not positive. Since the reduced graph is complete, all entries of the matrix representation of $D_{\lambda,\Vo}$ will, in general, be nonzero. For a derivation of this matrix, see Example~\ref{ex:nw-4-3-matrix}.
  \begin{figure}[ht]
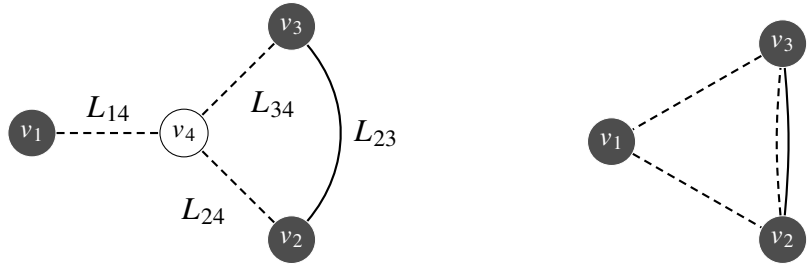

    \centering

    \caption{Graph with four edges and outer vertices $\Vo=\{v_1,v_2,v_3\}$ of three nodes and corresponding reduced graph.}
    \label{fig:nw-4-3}
  \end{figure}

  \begin{figure}[ht]
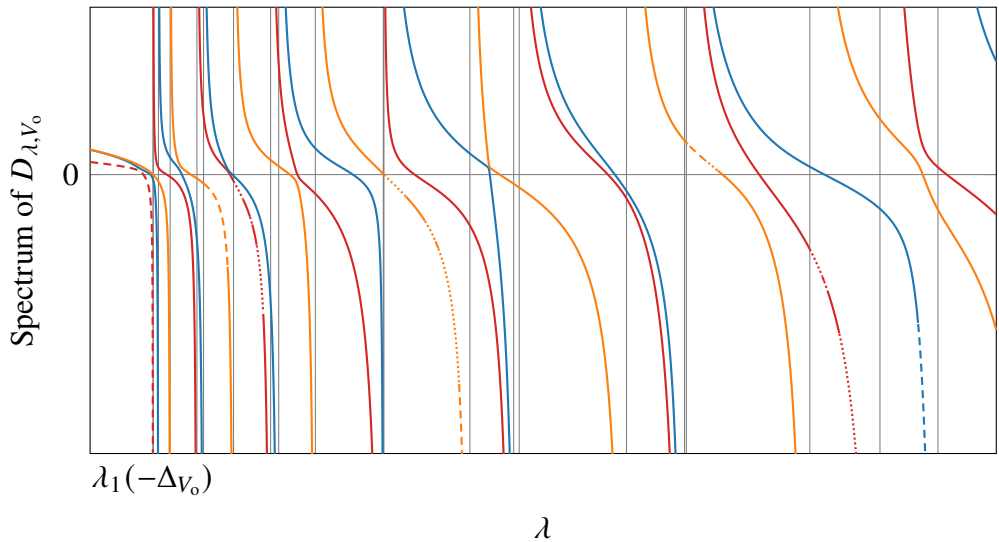

    \centering
    %
    \caption{Spectrum of $D_{\lambda,\Vo}$ corresponding to the graph in Figure~\ref{fig:nw-4-3}.}
    \label{fig:nw-4-3-spectrum}
  \end{figure}
\end{example}

\begin{example}
  \label{ex:star-graph}
  We consider a star graph with vertices $\Vo=\{v_1,v_2,v_3,v_4,v_5\}$ connecting to the center $\Vi=\{v_6\}$. As a special kind of tree, removing any of the edges disconnects it. However, the reduced graph is a complete graph. It stays connected by removing any of the edges, see Figure~\ref{fig:nw-6-1-graph}. As expected, we observe regions of $\lambda>\lambda_1(-\Delta_{\Vo})$ where the Dirichlet-to-Neumann semigroup is positive, or eventually positive without being positive. See Figure~\ref{fig:nw-6-1-spectrum}.
  \begin{figure}[ht]
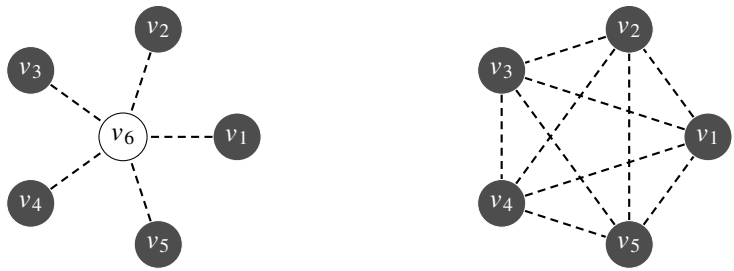

    \centering
    \tikzset{declare function={d=1.5;}}

    \caption{Star graph with center the inner vertex and corresponding reduced graph.}
    \label{fig:nw-6-1-graph}
  \end{figure}

  If we consider the graph in Figure~\ref{fig:nw-6-1-graph}, but with set of outer notes $V$, then removing any edge makes the graph disconnected, so it is a tree. According to Theorem~\ref{thm:main-result}~\ref{thm:main:tree} the Dirichlet-to-Neumann semigroup is either positive or not eventually strongly positive. The spectrum with the positivity properties is illustrated in Figure~\ref{fig:nw-6-0-spectrum}.

  \begin{figure}[ht]
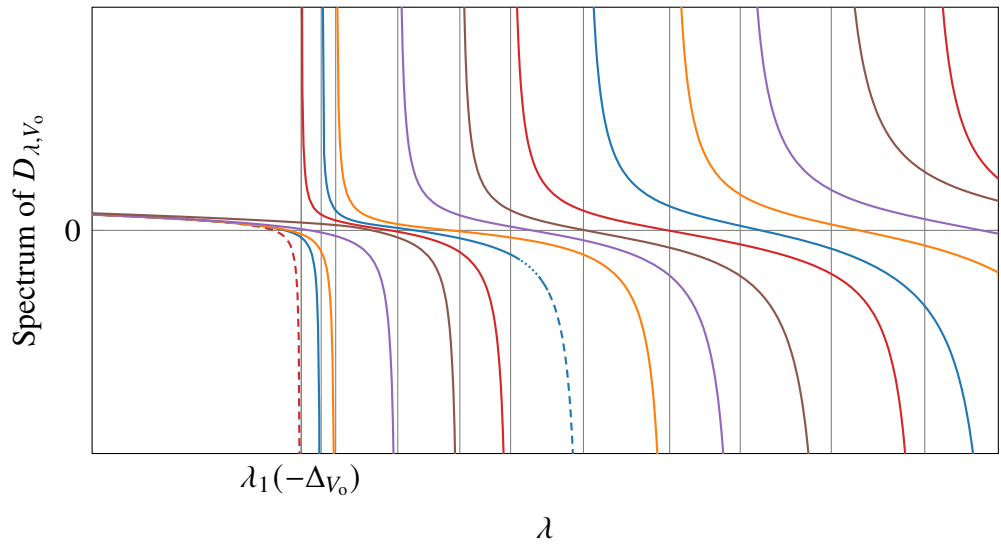

    \centering
    %
    \caption{Spectrum for star graph in Figure~\ref{fig:nw-6-1-graph}.}
    \label{fig:nw-6-1-spectrum}
  \end{figure}

  \begin{figure}[ht]
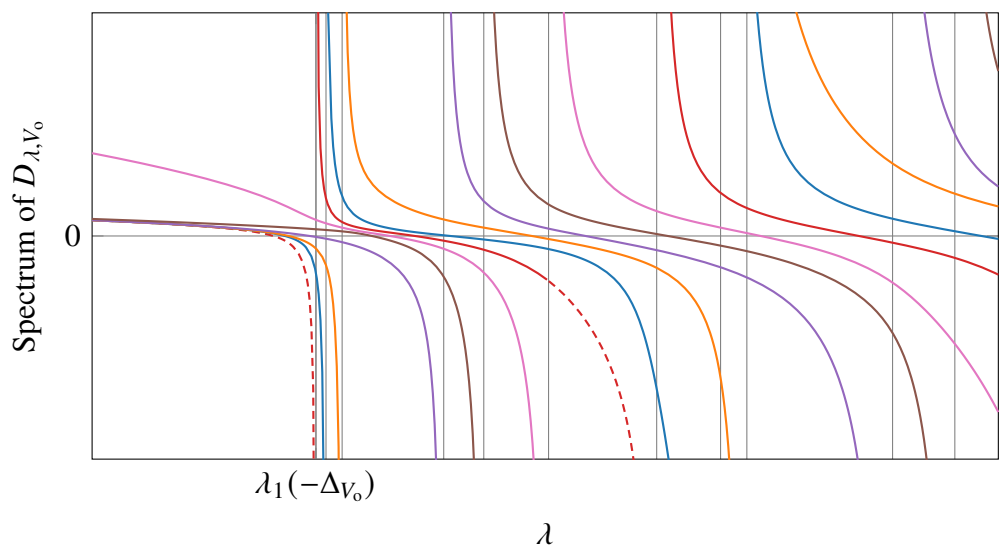

    \centering
    %
    \caption{Spectrum for graph in Figure~\ref{fig:nw-6-1-graph}, but with $\Vo=V$.}
    \label{fig:nw-6-0-spectrum}
  \end{figure}
\end{example}

\clearpage

\section{The Dirichlet-to-Neumann operator}
\label{sec:dtn-operator}
The purpose of this section is to give a careful construction of the Dirichlet-to-Neumann operator and other mostly well known background material. For background on quantum graphs we refer to \cite{berkolaiko:13:iqg,mugnolo:14:sme,kurasov:24:sgg} or \cite{mugnolo:21:wmg}. A discussion of Dirichlet-to-Neumann-type operators appears in several of the main monographs on differential equations on graphs, see for instance \cite[Section~3.5]{berkolaiko:13:iqg} or \cite[Section~6.6.1]{mugnolo:14:sme}. We nevertheless give a complete construction including the features that are relevant for our discussion on positivity properties, in particular the proof of Theorem~\ref{thm:main-result}\ref{thm:main:basic}. We keep using the assumptions on the graph $G=(E,V)$ from the introduction.

\subsection{Function spaces on graphs}
In order to do analysis on the edges of the graph $G$ we identify each edge $e \in E$ with the real interval $[0,L_e]$, where the left vertex $\el$ is identified with the end point $0$ and the right vertex $\er$ is identified with the end point $L_e$ of the interval. We define the \emph{$L^2$-space} on the graph $G$ by
\begin{equation*}
  L^2(G) := \bigoplus_{e \in E} L^2\bigl([0,L_e]\bigr)
  \quad
  \text{with norm}
  \quad
  \|(f_e)\|_2 := \left(\sum_{e\in E} \|f_e\|_2^2\right)^{1/2},
\end{equation*}
which clearly makes it into a Hilbert space. Similarly, following \cite{berkolaiko:13:iqg}, for $k\geq 1$ we define the \emph{Sobolev spaces} $\tilde H^k(G)$ on the graph $G$ by
\begin{equation*}
  \tilde H^k(G) := \bigoplus_{e \in E} H^k\bigl([0,L_e]\bigr)
  \quad
  \text{with norm}
  \quad
  \|(f_e)\|_{H^k} := \left(\sum_{e\in E} \|f_e\|_{H^k}^2\right)^{1/2},
\end{equation*}
which are also Hilbert spaces. We furthermore define the set of continuous functions on $G$ to be the set of continuous functions on the edges continuous across the vertices, that is,
\begin{equation*}
  C(G):=\left\{(f_e)_{e\in E}\in\bigoplus_{e\in E}C\left([0,L_e]\right)\colon
  f_e(v)=f_{\tilde e}(v)\text{ if }v\in V\text{ meets }e,\tilde e\in E\right\}
\end{equation*}
with norm
\begin{equation*}
  \|f\|_{C(G)}:=\max_{e\in E}\left\{\sum_{e\in E} \|f_e\|_{C([0,L_e])}\right\}.
\end{equation*}
Then $C(G)$ is a Banach space. By standard Sobolev embedding theorems $f_e\in C^{k-1}([0,L_e])$ for all $f=(f_e)\in\tilde H^k(G)$ and $k\geq 1$, see for instance \cite[Theorem~8.8]{brezis:11:fa}. Thus it makes sense to define
\begin{equation*}
  H^1(G):=C(G)\cap \tilde H^1(G).
\end{equation*}

\subsection{Homogeneous boundary value problems on graphs}
We can now introduce the Laplace operator on a graph with its maximal domain.

\begin{definition}[Maximal Laplace operator]
  The \emph{maximal Laplace operator}
  \begin{equation*}
    \Delta_{\max}\colon  D(\Delta_{\max})\subseteq L^2(G) \to L^2(G)
  \end{equation*}
  on $G$ is given by
  \begin{equation*}
    D(\Delta_{\max}) := C(G)\cap \tilde H^2(G),\quad
    \Delta_{\max}(f_e)_{e\in E} := (f_e'')_{e\in E},
  \end{equation*}
  for all $f=(f_e)_{e\in E}\in D(\Delta_{\max})$, where $f''$ denotes the second weak derivative of a function $f$.
\end{definition}

We next give a precise definition of the trace and the outer normal derivative of a function $f$. Let us begin with the trace.

\begin{definition}[Trace map]
  \label{def:trace}
  The \emph{trace map} $\trace\colon H^1(G) \to \bbR^V$ is defined by
  \begin{align*}
    (\trace f)(v) :=
    \begin{cases}
      f_e(0)   & \text{if } v = \el \\
      f_e(L_e) & \text{if } v = \er
    \end{cases}
  \end{align*}
  for all $f = (f_e) \in H^1(G)$ and all $v \in V$.
\end{definition}

Note that $\trace$ is well-defined since, as $G$ is connected, every vertex is connected to at least one edge, and since $\trace(f)(v)$ does not depend on $e$ since $H^1 (G)$ embeds into $C(G)$. Next, we consider outer normal derivatives of functions $f=(f_e)_{e\in E} \in \tilde H^2(G)$. If $f_e \in H^2(0,L_e)$, then the outer normal derivative of $f_e$ in $\el$ is defined as $-f_e'(0)$ and the outer normal derivative of $f$ in $\er$ is defined as $f_e'(L_e)$. This makes sense since $H^2\bigl((0,L_e)\bigr)\subseteq C^1\bigl([0,L_e]\bigr)$. To obtain the outer normal derivative of $f = (f_e)$ in a vertex $v$ we sum up all the outer normal derivatives in $v$ of all the functions $f_e$ for which $e$ has $v$ as an endpoint. More precisely, we make the following definition.
\begin{definition}[Outer normal derivative]
  The map $\nu\colon D(\Delta_{\max}) \to \bbR^V$ given by
  \begin{equation*}
    (\nu f)(v) := -\sum_{\el = v}f_e'(0) + \sum_{\er = v} f_e'(L_e)
  \end{equation*}
  for all $f = (f_e) \in D(\Delta_{\max})$ and all $v \in V$ is called the \emph{outer normal derivative mapping}. For each $f \in D(\Delta_{\max}$ and each $v\in V$ the number $(\nu f)(v)$ is called the \emph{outer normal derivative of $f$ at $v$}.
\end{definition}
This definition may readily be checked to be independent of the orientation of the edges. To be able to discuss the various boundary value problems we introduce the sesquilinear form
\begin{equation}
  \label{eq:a-form}
  a\colon H^1(G) \times H^1(G) \to \bbC, \quad
  a(f,g) = \sum_{e \in E} \int_0^{L_e} f_e' \overline{g_e'} \, dx.
\end{equation}
If $f\in\tilde H^2(G)$ and $g\in H^1(G)$, then an integration by parts yields
\begin{equation}
  \label{eq:a-form-integrated-by-parts}
  a(f,g)=-\sum_{e\in E}\int_0^{L_e}f_e''(x)\overline{g_e(x)}\,dx+\sum_{v\in V}\overline{g(v)}(\nu f)(v).
\end{equation}
We also introduce the space
\begin{equation}
  \label{eq:HVo}
  H^1_\Vo(G) := \{f \in C(G)\cap H^1(G)\colon (\trace f)|_\Vo = 0\}.
\end{equation}

Let now $\Vo \subseteq V$ be the non-empty subset of outer vertices and $\Vi$ its complementary set of inner vertices of $G$ as introduced in Section~\ref{sec:introduction}. With that we can give a proper definition of the Laplace operator $\Delta_{\Vo}$ with boundary conditions.
\begin{definition}[Laplace operator with boundary conditions]
  \label{def:laplace-Vo}
  The Laplace operator $\Delta_\Vo\colon D(\Delta_\Vo)\subseteq L^2(G)\to L^2(G)$ with Dirichlet boundary conditions on $\Vo$ and Neumann--Kirchhoff conditions on $\Vi$ is defined by
  \begin{align*}
    D(\Delta_\Vo) & := \{f \in D(\Delta_{\max})\colon  (\gamma f)|_\Vo = 0 \text{ and } (\nu f)|_\Vi = 0\}, \\
    \Delta_\Vo f  & := \Delta_{\max} f
  \end{align*}
  for all $f\in D(\Delta_{\Vo})$.
\end{definition}
Then $\Delta_{\Vo}$ has the following properties.
\begin{proposition}
  \label{prop:simple-properties-of-laplacian}
  The operator $\Delta_\Vo$ on $L^2(G)$ from Definition~\ref{def:laplace-Vo} is self-adjoint, has compact resolvent and
  \begin{equation}
    \label{eq:lambda-1-rayleigh}
    \inf\sigma(-\Delta_{\Vo})
    =\lambda_1(-\Delta_{\Vo})
    =\inf_{\substack{w\in H_{\Vo}^1(G)\\w\neq 0}}\frac{a_0(w,w)}{\|w\|_{L^2(G)}}>0.
  \end{equation}
  Moreover, $\Delta_{\lambda,\Vo}$ generates a positive semigroup on $L^2(G)$.
\end{proposition}
\begin{proof}
  Consider the form $a$ from \eqref{eq:a-form} restricted to the space $H^1_\Vo(G) \times H^1_\Vo(G)$. Clearly, $a$ is a symmetric, accretive sesquilinear form on $H_{Vo}^1(G)$. It follows from \eqref{eq:a-form-integrated-by-parts} that the operator induced by $a$ on $L^2(G)$ coincides with $-\Delta_\Vo$. Hence, $\Delta_\Vo$ is self-adjoint and $\sigma(-\Delta_\Vo) \subseteq [0,\infty)$. Since $D(\Delta_\Vo)$ is a subset of $\tilde H^2(G)$ which compactly embeds into $L^2(G)$, it follows that $\Delta_\Vo$ has compact resolvent. In particular, $\sigma(\Delta_\Vo)$ consists only of eigenvalues and therefore, it remains to show that $0$ is not an eigenvalue of $\Delta_\Vo$.

  If $f\in D(\Delta_{\Vo})$ and $a(f,f)=0$, then $f'=0$ on each edge of $G$. As $G$ is connected, $f$ is therefore constant on the entire graph $G$. Since $f$ satisfies Dirichlet boundary conditions on the non-empty set $\Vo$, we conclude that $f = 0$. Hence zero is not an eigenvalue of $\Delta_{\Vo}$. The identity \eqref{eq:lambda-1-rayleigh} is the standard Rayleigh quotient (variational) characterisation of the first eigenvalue $\lambda_1(-\Delta_{\Vo})$ of a self-adjoint operator with compact resolvent applied to this case.

  To prove the last statement note that $a(f^+,f^-)=0$, where $f^+:=\max\{f,0\}$ and $f^-:=\max\{-f,0\}$. According to the first Beurling--Deny criterion the semigroup generated by $\Delta_{\Vo}$ is positive, see for instance \cite[Theorem~2.6]{ouhabaz:05:ahe}.
\end{proof}

\begin{remark}
  It follows from \eqref{eq:lambda-1-rayleigh} that $\lambda_1(-\Delta_{\Vo})$ is monotone in $\Vo$. In particular, $\lambda_1(-\Delta_{\Vo})\geq\lambda_1(-\Delta_V)$.
\end{remark}

In the case where $\Vo = V$ and $\Vi = \emptyset$ it is easy to compute the spectrum of $\Delta_V$.

\begin{lemma}
  \label{lem:spectrum-Delta-V}
  We have $\sigma(-\Delta_V) = \left\{\left(\dfrac{\pi k}{L_e}\right)^2\colon k \in \bbN, \; e \in E\right\} \subseteq (0,\infty)$.
\end{lemma}
\begin{proof}
  Let $\lambda \in \bbC$. As every function in $D(\Delta_V)$ fulfils Dirichlet boundary conditions on all vertices, the eigenvalue problem $-\Delta_V f = \lambda f$ has a solution if and only if the eigenvalue problem for minus the Dirichlet Laplacian has a solution on at least one of the intervals $[0,L_e]$ ($e \in E$). Hence, $\sigma(-\Delta_V)$ has the claimed form since $\left\{\left(\dfrac{\pi k}{L}\right)^2\colon k \in \bbN\right\}$ is the spectrum of minus the Dirichlet Laplacian on an interval of length $L$.
\end{proof}
\subsection{Definition of the Dirichlet-to-Neumann operator}
Let $\lambda \in \bbC$ and $x\in C^{\Vo}$. As a first step to define the \emph{Dirichlet-to-Neumann operator} we look at the problem of finding a function $f\in D(\Delta_{\max})$ satisfying the inhomogeneous Dirichlet problem \eqref{eq:boundary-value-problem}. The following proposition characterises the solvability of this boundary value problem in terms of the corresponding eigenvalue problem for $\Delta_\Vo$.

\begin{proposition}
  \label{prop:boundary-value-problem-and-eigenvalue-problem}
  Let $\lambda \in \bbC$. Then the following assertions are equivalent:
  \begin{enumerate}[label={\normalfont (\roman*)}]
  \item For every $x \in \bbC^\Vo$ the problem~\eqref{eq:boundary-value-problem} has a unique solution $f \in D(\Delta_{\max})$.
  \item $\lambda \not\in \sigma(-\Delta_\Vo)$.
  \end{enumerate}
\end{proposition}
We first consider the case $\lambda=0$. This appears in \cite[page~44]{mugnolo:14:sme}, but without proof or reference, so for completeness we include a proof based on standard variational tools.
\begin{lemma}
  \label{lem:boundary-value-problem-special-case}
  Let $\lambda = 0$ and let $x \in \bbC^\Vo$. Then the problem~\eqref{eq:boundary-value-problem} has a solution $f \in D(\Delta_{\max})$. If $x\geq 0$, then there exists a solution $f\geq 0$.
\end{lemma}
\begin{proof}
  Let $x \in \bbC^\Vo$ and set $M:=\{f \in C(G)\cap H^1 (G): (\trace f)|_{\Vo} = x\}$. We define a functional
  \begin{equation*}
    \mathcal{E}: M \to [0,\infty),\, f\mapsto \dfrac{1}{2}a(f,f),
  \end{equation*}
  where $a$ is given by \eqref{eq:a-form}. We claim that there exists $f^\ast \in M$ with
  \begin{equation*}
    \mathcal E(f^\ast)=\min_{f \in M} \mathcal{E}(f).
  \end{equation*}
  To do so let $h\in M$ and let $(f_n)_{n\in\mathbb N}$ be a minimising sequence in $M$. Note that by \eqref{eq:lambda-1-rayleigh} and the Cauchy-Schwarz inequality we have that
  \begin{equation*}
    \|f_n-h\|_{L^2(G)}^2
    \leq \lambda_1(-\Delta_{\Vo})a(f-h,f-h)
    \leq\lambda_1(-\Delta_{\Vo})\left(\mathcal E(f_n)+\mathcal E(h)\right)
  \end{equation*}
  for every $n\in\mathbb N$. Hence, $(f_n)_{n\in\mathbb N}$ is bounded in $H^1(G)$ and thus has a weak limit point $f^\ast \in H^1 (G)$. Then,
  \begin{equation*}
    \mathcal{E}(f^\ast) \leq \liminf_{n\to \infty} \mathcal{E}(f_n).
  \end{equation*}
  Since the trace $\trace$, having finite dimensional range, is a compact operator, we have $x = (\trace f_n)|_{\Vo} \to (\trace f^\ast)|_{\Vo}$ as $n\to\infty$ and thus $f^\ast \in M$. If $x\geq 0$, then also $(f^\ast)^+:=\max\{f^\ast,0\}\in M$ and $a((f^\ast)^+,(f^\ast)^+)\leq a(f^\ast,f^\ast)$. Hence, if $x\geq 0$, then there exists a positive minimizer.

  We next prove that the minimizer $f^\ast$ is a solution of \eqref{eq:boundary-value-problem}. Since $f \in M$ if and only if $f+g \in M$ for every $g \in H^1_\Vo (G)$, the Euler--Lagrange equation associated with $\mathcal{E}$ on $M$ reads $a(f^\ast,g) = 0$ for all $g \in H^1_\Vo(G)$. An argument as in Proposition~\ref{prop:simple-properties-of-laplacian} now yields $f^\ast \in \tilde H^2 (0,L_e)\cap H^1(G)$. It follows from \eqref{eq:a-form-integrated-by-parts} that $(f_e^\ast)'' = 0$ for all $e \in E$ and that $(\nu f^\ast)(v)=0$ for all $v \in \Vi$. Thus $f^\ast \in D(\Delta_{\max})$ solves \eqref{eq:boundary-value-problem}.
\end{proof}
Now we can prove Proposition~\ref{prop:boundary-value-problem-and-eigenvalue-problem}.

\begin{proof}[Proof of Proposition~\ref{prop:boundary-value-problem-and-eigenvalue-problem}]
  (i) $\Rightarrow$ (ii): If $\lambda\in\sigma(-\Delta_{\Vo})$, then $\lambda$ is an eigenvalue and thus \eqref{eq:boundary-value-problem} has infinitely many solutions for $x=0\in \bbC^{\Vo}$, namely all scalar multiples of the eigenfunction.

  (ii) $\Rightarrow$ (i): Let $\lambda\in\bbC\setminus\sigma(-\Delta_\Vo)$ and let $x \in \bbC^\Vo$. We have to show that problem~\eqref{eq:boundary-value-problem} has a unique solution $f \in D(\Delta_{\max})$. To prove the uniqueness let $f_1,f_2\in D(\Delta_{\max})$ be solutions of \eqref{eq:boundary-value-problem}. Then $f := f_1 - f_2\in D(\Delta_{\Vo})$ satisfies $\lambda f = -\Delta_{\max}f = -\Delta_\Vo f$. As $\lambda \not\in \sigma(-\Delta_\Vo)$ it follows that $f_1=f_2$.

  To prove the existence we make use of Lemma~\ref{lem:boundary-value-problem-special-case} asserting that problem~\eqref{eq:boundary-value-problem} has a solution in $g\in D(\Delta_{\max})$ in case that $\lambda = 0$. As $\lambda\not\in\sigma(-\Delta_\Vo)$, we can define
  \begin{equation*}
    f := g - \lambda R(\lambda,-\Delta_{\Vo})g \in D(\Delta_{\max}).
  \end{equation*}
  As $R(\lambda,-\Delta_\Vo)g\in D(\Delta_\Vo)$ we have $(\trace f)|_{\Vo}=(\trace g)_{\Vo}=x$ and $(\nu f)|_{\Vi}=(\nu g)|_{\Vi}=0$. This shows that $f$ fulfils the boundary conditions in problem~\eqref{eq:boundary-value-problem}. Moreover, we have
  \begin{equation*}
    \Delta_{\max}f
    = \Delta_{\max}g - \lambda \Delta_\Vo R(\lambda,-\Delta_\Vo)g
    = -\lambda \left(g - \lambda R(\lambda,-\Delta_\Vo)g \right)
    = -\lambda f.
  \end{equation*}
  Hence, $f$ is a solution of~\eqref{eq:boundary-value-problem}.
\end{proof}

With this knowledge the Dirichlet-to-Neumann operator can be established.
\begin{definition}
  \label{def:dirichlet-to-neumann-operator}
  Let $\lambda \in \bbR \setminus \sigma(-\Delta_\Vo)$. The \emph{Dirichlet-to-Neumann operator} on $G$ with potential $\lambda$ and outer vertex set $\Vo$ is defined as
  \begin{equation*}
    D_{\lambda,\Vo}\colon \bbC^\Vo \to \bbC^\Vo,\quad
    x \mapsto (\nu f)|_{\Vo},
  \end{equation*}
  where $f \in D(\Delta_{\max})$ is the unique solution of the problem~\eqref{eq:boundary-value-problem} which exists according to Proposition~\ref{prop:boundary-value-problem-and-eigenvalue-problem}.
\end{definition}
It is immediate from the definition and \eqref{eq:a-form-integrated-by-parts} that $D_{\lambda,\Vo}x = y$ if and only if
\begin{equation}
  \label{eq:dtn-bilinear-form}
  a_\lambda(f_x,g)
  := a(f_x,g) - \lambda (f_x,g)_{L^2 (G)}
  = (y, (\trace g)|_{\Vo})_{\bbC^\Vo}
\end{equation}
for all $g \in H^1 (G)$, where $f_x\in D(\Delta_{\max})$ is the unique solution of \eqref{eq:boundary-value-problem}. This says exactly that $D_{\lambda,\Vo}$ is the operator associated with the \emph{pair} $(a_\lambda, (\trace \,\cdot\,)|_{\Vo})$ in the sense of \cite[Theorem~2.1]{arendt:12:fei}; see also \cite[Section~3]{arendt:11:dno}. Since $a_\lambda$ is sesquilinear and closed and $(\trace\,\cdot\,)|_{\Vo}$ is bounded linear it follows from \cite[Theorem~4.5]{arendt:14:dno} that $D_{\lambda,\Vo}$ is a self-adjoint and in particular real operator. As we can (and will) identify the Dirichlet-to-Neumann operator with a matrix $D_{\lambda,\Vo} \in \bbC^{\Vo \times \Vo}$ with respect to the natural basis of $\bbC^{\Vo}$, we summarise this in the following proposition.

\begin{proposition}
  \label{prop:dtn-operator}
  Let $\lambda \in \bbR \setminus \sigma(-\Delta_\Vo)$. Then the Dirichlet-to-Neumann operator $D_{\lambda,\Vo}$ is self-adjoint and real, that is, $D_{\lambda,\Vo}$ is a symmetric matrix in $\bbR^{\Vo \times \Vo}$.
\end{proposition}

\section{Matrix representation of  the Dirichlet-to-Neumann operator}
\label{sec:matrix-representation}

This section is essentially devoted to the proof of Proposition~\ref{prop:main}. We will prove \ref{item:main:formula-for-dtn-operator-with-inner-vertices} first, and then \ref{item:main:adjacency-class} afterwards. That is, we will start by deriving the Schur complement formula \eqref{eq:DN-Vo}.

We start with the case $\Vo = V$, where we can determine a matrix representation of the Dirichlet-to-Neumann operator explicitly by computing it on every edge and then adding them all up; this is exactly \eqref{eq:DN-V}.

We first order the elements in $V=(v_1,v_2,\dots,v_n)$ and identify each element with the corresponding standard basis vector of $\bbC^n$. We also identify the edge between $v_k$ and $v_j$ by $(k,j)$. According to \cite[Section~3]{daners:14:nps} or \cite[Section~3.5.1]{berkolaiko:13:iqg}, for admissible $\lambda \in \bbR$ the operator associated with the interval of length $L_{kj}:=L_{v_k,v_j}\in(0,\infty)$ ($k\neq j$) can be represented by the matrix
\begin{equation*}
  M_{k,j}(\lambda):=
  \begin{bmatrix}
    0      & \dots & 0                    & \dots  & 0                    & \dots & 0      \\
    \vdots &       & \vdots               &        & \vdots               &       & \vdots \\
    0      & \dots & \alpha_{jk}(\lambda) & \dots  & -\beta_{jk}(\lambda) & \dots & 0      \\
    \vdots &       & \vdots               & \ddots & \vdots               &       & \vdots \\
    0      & \dots & -\beta_{kj}(\lambda) & \dots  & \alpha_{kj}(\lambda) & \dots & 0      \\
    \vdots &       & \vdots               &        & \vdots               &       & \vdots \\
    0      & \dots & 0                    & \dots  & 0                    & \dots & 0      \\
  \end{bmatrix}
  ,
\end{equation*}
where
\begin{equation}
  \label{eq:alpha}
  \alpha_{kj}(\lambda):=
  \begin{cases}
    \dfrac{\sqrt{\lambda}\*\cos\left(\sqrt{\lambda}L_{kj}\right)}
    {\sin\left(\sqrt{\lambda}L_{kj}\right)}
                      & \text{if $\lambda\neq 0$,} \\
    \dfrac{1}{L_{kj}} & \text{if $\lambda=0$,}
  \end{cases}
\end{equation}
and
\begin{equation}
  \label{eq:beta}
  \beta_{kj}(\lambda):=
  \begin{cases}
    \dfrac{\sqrt{\lambda}}{\sin\left(\sqrt{\lambda}L_{kj}\right)}
                      & \text{if $\lambda\neq 0$,} \\
    \dfrac{1}{L_{kj}} & \text{if $\lambda=0$.}
  \end{cases}
\end{equation}
For convenience we set $\alpha_{kj}=\beta_{kj}=0$ if $L_{kj}=\infty$ or $k=j$. We note that for $\lambda<0$ we can write
\begin{equation*}
  \beta_{kj}(\lambda)
  =\frac{\sqrt{\lambda}}{\sin\left(\sqrt{\lambda}L_{kj}\right)}
  =\frac{i\sqrt{-\lambda}}{\sin\left(i\sqrt{-\lambda}L_{kj}\right)}
  =\frac{i\sqrt{-\lambda}}{i\sinh\left(\sqrt{-\lambda}L_{kj}\right)}
  =\frac{\sqrt{-\lambda}}{\sinh\left(\sqrt{-\lambda}L_{kj}\right)}
\end{equation*}
and similarly
\begin{equation*}
  \alpha_{kj}(\lambda)
  =\frac{\sqrt{-\lambda}\cosh\left(\sqrt{-\lambda}L_{kj}\right)}
  {\sinh\left(\sqrt{-\lambda}L_{kj}\right)}.
\end{equation*}
Furthermore, $\lambda=0$ is a removable singularity of $\alpha_{kj}$ and $\beta_{kj}$ since
\begin{equation*}
  \lim_{\lambda\to 0}\beta_{kj}(\lambda)
  =\frac{1}{L_{kj}}\lim_{\lambda\to 0}
  \frac{\sqrt{\lambda}L_{kj}}{\sin\left(\sqrt{\lambda}L_{kj}\right)}
  =\dfrac{1}{L_{kj}}
  =\lim_{\lambda\to 0}\alpha_{kj}(\lambda).
\end{equation*}
Note that by definition of the (principal) square root on $\mathbb C$, and since $\sin$ is odd and $\cos$ is even, it follows that $\lambda\mapsto\beta_{kj}(\lambda)$ and $\lambda\mapsto\alpha_{kj}(\lambda)$ are analytic functions apart from isolated singularities at $(\pi k/L_{kj})^2$ for $k\in\bbN$. It follows that
\begin{equation}
  \label{eq:dtn-operator}
  D_{\lambda,V}=\sum_{1\leq k<j\leq n}M_{kj}(\lambda)
\end{equation}
for all $\lambda\in\bbC\setminus\sigma(-\Delta_{\Vo})$. Equation~\eqref{eq:beta} also shows that $[D_{\lambda,V}]_{k,j}\neq 0$ for all $(k,j)\in E$. As $G$ is connected it follows from Lemma~\ref{lem:irreducible-matrix} that $D_{\lambda,V}$ is irreducible for all $\lambda\in\mathbb C\setminus\{-\Delta_{V}\}$. We summarise the result in the following proposition.

\begin{proposition}
  \label{prop:explicit-formula-for-dtn-operator}
  Let $\lambda \in \bbC \setminus \sigma(\Delta_V)$. Then, using the notation introduced above, the Dirichlet-to-Neumann operator $D_{\lambda,V}$ is given by the symmetric matrix with entries
  \begin{equation}
    \label{eq:dtn-operator-entries}
    [D_{\lambda,V}]_{kj}=
    \begin{cases}
      -\beta_{kj}(\lambda) & \text{if $k\neq j$,} \\
      \displaystyle\sum_{l=1}^n\alpha_{kl}(\lambda)
                           & \text{if $k=j$,}
    \end{cases}
  \end{equation}
  where $1\leq k,j\leq n$. Moreover, $D_{\lambda,V}$ is irreducible, with $[D_{\lambda,V}]_{kj}\neq 0$ for all $(k,j)\in E$. Finally, the map $\lambda\mapsto D_{\lambda,V}$ is analytic as a function of $\lambda\in\mathbb C\setminus\sigma(-\Delta_{V})$ with a pole of order one at every point in $\sigma(-\Delta_{V})$.
\end{proposition}
\begin{proof}
  We only need to show that $\lambda\in\sigma(-\Delta_V)$ are poles of order one. The singularities occur at $\sigma(-\Delta_V)$ as determined in Lemma~\ref{lem:spectrum-Delta-V}. We note that
  \begin{align*}
    \beta_{kj}(\lambda)\left(\lambda-\left(\frac{\pi k}{L_{kj}}\right)^2\right)
     & =\sqrt{\lambda}\left(\sqrt{\lambda}+\frac{\pi k}{L_{kj}}\right)
    \frac{\sqrt{\lambda}-\frac{\pi k}{L_{kj}}}{\sin\left(\sqrt{\lambda}L_{kj}\right)-\sin(\pi k)} \\
     & \to 2\left(\frac{\pi k}{L_{kj}}\right)^2\frac{1}{\cos(\pi k)}
    =2(-1)^k\left(\frac{\pi k}{L_{kj}}\right)^2
  \end{align*}
  as $\lambda\to\left(\frac{\pi k}{L_{kj}}\right)^2$, showing that $\left(\frac{\pi k}{L_{kj}}\right)^2$ is a pole of order one for $\beta_{kj}(\lambda)$. A similar statement holds for $\alpha_{kj}(\lambda)$.
\end{proof}
\begin{example}
  \label{ex:nw-4-4-matrix}
  As an example consider the graph shown in Figure~\ref{fig:nw-3-2} with $V=\{v_1,v_2,v_3,v_4\}$. The matrix representing $D_{\lambda,V}$ is given by
  \begin{equation}
    \label{eq:nw-4-4-matrix}
    D_{\lambda,V}=
    \begin{bmatrix}
      \alpha_{14} & 0                       & 0                       & -\beta_{14}                         \\
      0           & \alpha_{23}+\alpha_{24} & -\beta_{23}             & -\beta_{24}                         \\
      0           & -\beta_{32}             & \alpha_{32}+\alpha_{34} & -\beta_{34}                         \\
      -\beta_{41} & -\beta_{42}             & -\beta_{43}             & \alpha_{41}+\alpha_{42}+\alpha_{43} \\
    \end{bmatrix}
    .
  \end{equation}
\end{example}
With this background, we can now give the proof of our first result.

\begin{proof}[Proof of Proposition~\ref{prop:main}\ref{item:main:formula-for-dtn-operator-with-inner-vertices}] The aim is to use $D_{\lambda,V}$ to compute $D_{\lambda,\Vo}$ in the case $\Vo$ is a proper subset of $V$. We assume that $\Vo$ has $m<n$ elements. Fix $\lambda\not\in\sigma(-\Delta_V)$ and consider the block matrix \eqref{eq:DN-V} with $A\in\bbC^{m\times m}$, $C\in\bbC^{(n-m)\times(n-m)}$ and $B\in\bbC^{(n-m)\times m}$. Suppose that we are given Dirichlet data $x\in \mathbb C^m$ on $\Vo$. By Proposition~\ref{prop:boundary-value-problem-and-eigenvalue-problem} the boundary value problem \eqref{eq:boundary-value-problem} has a unique solution $f\in H^1(G)$. Setting $y:=(\trace f)|_{\Vi}$ we conclude that
  \begin{equation*}
    -D_{\lambda,V}(\trace f)=
    \begin{bmatrix}
      A & B^T \\
      B & C
    \end{bmatrix}
    \begin{bmatrix}
      x \\y
    \end{bmatrix}
    =
    \begin{bmatrix}
      -(\nu f)|_{\Vo} \\0
    \end{bmatrix}
    .
  \end{equation*}
  The above matrix equation is equivalent to the system of equations
  \begin{equation}
    \label{eq:DN-system}
    \begin{aligned}
      Ax+B^Ty & =-(\nu f)|_{\Vo}=-D_{\lambda,\Vo}x, \\
      Bx+Cy   & =0.
    \end{aligned}
  \end{equation}
  This shows that if $\lambda\in\bbC\setminus\sigma(D_{\lambda,V})$ and $x\in\bbC^m$, then given $x\in\bbC^m$ there exists a unique $y\in\bbC^{n-m}$ such that $-Cy=Bx$. In particular, this implies that
  \begin{equation}
    \label{eq:im-BC}
    \im(B)\subseteq\im(C).
  \end{equation}
  To show that $\ker(C)=\{0\}$ and hence that $C$ is indeed invertible, we fix $y_0\in\ker(C)$. Then, by the symmetry of $C$, $y_0\in\ker(C^T)$. It follows that $0=z^TC^Ty_0=(Cz)^Ty_0$ for all $z\in\mathbb C^{n-m}$. Using \eqref{eq:im-BC} we conclude that $0=(Bx)^Ty_0=x^T(B^Ty_0)$ for all $x\in\bbC^{m}$. Hence, $B^Ty_0=0$, showing that $\ker(C)\subseteq\ker(B^T)$. As a consequence, for every $y_0\in\ker(C)$, we have that
  \begin{equation*}
    -D_{\lambda,V}
    \begin{bmatrix}
      0 \\y_0
    \end{bmatrix}
    =
    \begin{bmatrix}
      A & B^T \\
      B & C
    \end{bmatrix}
    \begin{bmatrix}
      0 \\y_0
    \end{bmatrix}
    =
    \begin{bmatrix}
      0 \\0
    \end{bmatrix}
    =
    \begin{bmatrix}
      -(\nu f)|_{\Vo} \\0
    \end{bmatrix}
    ,
  \end{equation*}
  where $f$ is the unique solution of~\eqref{eq:boundary-value-problem} with $x=0$, that is, the unique $f \in D(\Delta_{\max})$ such that $\Delta_{\max} f + \lambda f = 0$, $(\gamma f)|_{\Vo} = x$ and $(\gamma f)|_{\Vi} = y_0$. By Proposition~\ref{prop:boundary-value-problem-and-eigenvalue-problem} $f=0$ is the only solution and thus $y_0=(\trace f)|_{\Vi}=0$. Hence, $\ker(C)=\{0\}$ and thus $C$ is invertible for every $\lambda\in\bbC\setminus\sigma(-\Delta_{\Vo})$. Hence, by \eqref{eq:DN-system} we can always write $y=(-C)^{-1}Bx$ and thus $-D_{\lambda,\Vo}(x)=-(\nu f)|_{\Vo}=Ax+B^Ty=Ax+B^T(-C)^{-1}Bx$ is the Schur complement of the block $C$ in \eqref{eq:DN-V}. This proves Proposition~\ref{prop:main}\ref{item:main:formula-for-dtn-operator-with-inner-vertices}.
\end{proof}

\begin{example}
  \label{ex:nw-4-3-matrix}
  As an example we can apply the formula to the graph in Example~\ref{ex:nw-4-3} and Figure~\ref{fig:nw-4-3} with $\Vo=\{v_1,v_2,v_3\}$. We have computed $D_{\lambda,V}$ in Example~\ref{ex:nw-4-4-matrix}. Hence the blocks in the decomposition \eqref{eq:DN-V} are
  \begin{align*}
    A & =
    \begin{bmatrix}
      -\alpha_{14} & 0                          & 0                          \\
      0            & -(\alpha_{23}+\alpha_{24}) & \beta_{23}                 \\
      0            & \beta_{32}                 & -(\alpha_{32}+\alpha_{34})
    \end{bmatrix}
    ,                                            \\
    B & =
    \begin{bmatrix}
      \beta_{41} & \beta_{42} & \beta_{43}
    \end{bmatrix}
    ,                                            \\
    C & =-[\alpha_{41}+\alpha_{42}+\alpha_{43}].
  \end{align*}
  We then have
  \begin{equation*}
    B^T(-C)^{-1}B
    =\frac{1}{\alpha_{41}+\alpha_{42}+\alpha_{43}}
    \begin{bmatrix}
      \beta_{14}\beta_{41} & \beta_{14}\beta_{42} & \beta_{14}\beta_{43} \\
      \beta_{24}\beta_{41} & \beta_{24}\beta_{42} & \beta_{24}\beta_{43} \\
      \beta_{34}\beta_{41} & \beta_{34}\beta_{42} & \beta_{34}\beta_{43} \\
    \end{bmatrix}
  \end{equation*}
  and so we can easily write down $D_{\lambda,\Vo}$ using formula~\eqref{eq:DN-Vo}.
\end{example}

We observe that due to \eqref{eq:dtn-operator} it follows that $D_{\lambda,V}\in\mathcal A(G)$, that is, the assertion of Proposition~\ref{prop:main}\ref{item:main:adjacency-class} holds for the special case $\Vo=V$. For the rest of this section, we will study the structure of the matrix $D_{\lambda,\Vo}$ more carefully, and in particular prove that $D_{\lambda,\Vo}\in\mathcal A(G_0)$ when $\Vo$ is a proper subset of $V$, which is exactly Proposition~\ref{prop:main}\ref{item:main:adjacency-class}. For that purpose we study the block decomposition \eqref{eq:DN-V} of the adjacency matrix $A_G$ of $G$. Consider the induced sub-graphs $G[\Vo]$ and $G[\Vi]$. These graphs are not necessarily connected, but they consist of a finite union of connected components. By re-numbering the vertices in $\Vo$ and $\Vi$ and the corresponding edges we can assume without loss of generality that all vertices are consecutively numbered within the connected components, starting with those in $G[\Vo]$. With that choice of numbering, the matrices $A$ and $C$ have a diagonal block structure
\begin{equation}
  \label{eq:A-C-block}
  A=
  \begin{bmatrix}
    A_1    & 0      & \dots  & 0                    \\
    0      & A_2    & \ddots & \vdots               \\
    \vdots & \ddots & \ddots & 0                    \\
    0      & \dots  & 0      & A_{m_{\outerVertex}}
  \end{bmatrix}
  \qquad\text{and}\qquad
  C=
  \begin{bmatrix}
    C_1    & 0      & \dots  & 0                    \\
    0      & C_2    & \ddots & \vdots               \\
    \vdots & \ddots & \ddots & 0                    \\
    0      & \dots  & 0      & C_{m_{\innerVertex}}
  \end{bmatrix}
  ,
\end{equation}
with every diagonal block representing a connected component of the induced sub-graphs $G[\Vo]$ and $G[\Vi]$, respectively. In particular, each of these blocks is an irreducible matrix by Lemma~\ref{lem:irreducible-matrix} in the appendix.

Due to the connectedness of $G$, each of the connected components of $G[\Vo]$ connects to at least one connected component of $G[\Vi]$. As $G$ is connected, also the reduced graph $G_{\outerVertex}=(\Vo,E_{\outerVertex})$ from Definition~\ref{def:reduced-graph} is connected. The situation is illustrated in Figure~\ref{fig:graph-with-subgraphs}. The dashed lines show the edges providing connections between the connected components of the induced graphs $G[\Vo]$ and $G[\Vi]$. The corresponding reduced graph $G_{\outerVertex}$ is shown in Figure~\ref{fig:graph-with-subgraphs-reduced}. In that example $A$ has three irreducible blocks, and $C$ has two irreducible blocks. The following proposition in particular contains the assertions of Proposition~\ref{prop:main}\ref{item:main:adjacency-class} by applying it to $Q=D_{\lambda,\Vo}$.
\begin{proposition}
  \label{prop:properties-for-dtn-operator-with-inner-vertices}
  Let $G=(V,E)$ be a connected graph as in Section~\ref{sec:introduction} and let $\Vo$ and $\Vi$ be the sets of outer and inner vertices, respectively. Let $G_{\outerVertex}=(\Vo,E_{\outerVertex})$ be the reduced graph associated with $\Vo$. Suppose that $Q\in\mathcal A(G)$ and consider a block decomposition of $Q$ of the form \eqref{eq:DN-V}.
  \begin{enumerate}[label={\normalfont (\roman*)}]
  \item If $C$ is invertible and
    \begin{equation}
      \label{eq:M}
      M:=A+B^T(-C)^{-1}B
    \end{equation}
    is the Schur complement of $C$, then $M\in\mathcal A(G_{\outerVertex})$.
  \item Assume in addition that $Q_{kj}>0$ for all $(k,j)\in E$ and that $\spb(C)<0$. Then the Schur complement \eqref{eq:M} exists and $M_{rs}>0$ for all $(r,s)\in E_{\outerVertex}$. Moreover, given $v_r,v_s\in\Vo$ with $r\neq s$, we have that $M_{rs}=A_{rs}$ if and only if $v_r$ and $v_s$ are not connected by a path through a connected component of $\Vi$.
  \end{enumerate}
\end{proposition}
\begin{proof}
  We continue to use the above notation and setup and let $M$ be the Schur complement of $Q$ given by \eqref{eq:M}. Under the assumptions of (ii) we have $\spb(C)<0$ and thus $\spb(C_j)<0$ for each block matrix $C_j$. Hence, the irreducible matrix $C_j$ is invertible with $(-C_j)^{-1}\gg 0$ by Lemma~\ref{lem:positive-semigroups}, so $C$ is invertible as well.

  Let now $(r,s)\in E_{\outerVertex}$. If there is a connection between $v_r$ and $v_s$ through some connected component $\tilde G$ of $G[\Vi]$, then there exist vertices $v_j$ and $v_\ell$ in $\tilde G$ so that $(j,r)\in E$ and $(\ell,s)\in E$. In the example shown in Figure~\ref{fig:graph-with-subgraphs}, $v_1,v_2\in \Vo$ connect to vertex $v_8\in \Vi$, and vertex $v_3\in G[\Vo]$ connects to vertex $v_9\in \Vi$. Assume that $A_1$ and $A_2$ are the irreducible blocks corresponding to the connected components of $G[\Vo]$ containing $v_r$ and $v_s$, respectively. We assume the block $C_0$ of $C$ corresponding to the component $\tilde G$ of $G[\Vi]$ containing $v_j$ and $v_\ell$. The situation is shown in Figure~\ref{fig:L-blocks}. Under the assumption (ii) the entries $(j,r)$ and $(\ell,s)$ of $B$ are positive. As $(-C_0)^{-1}\gg 0$, it follows that
  \begin{equation*}
    [B^T(-C)^{-1}B]_{rs}>0.
  \end{equation*}
  Under the assumption (i), the entry $[B^T(-C)^{-1}B]_{rs}$ may or may not be zero. If there is no path from $v_r$ to $v_s$ through the component $\tilde G$ of $G[\Vi]$, then the $A_{rj}=0$ or $A_{sj}=0$ whenever $v_j\in \tilde G$. In either case
  \begin{equation*}
    [B^T(-C)^{-1}B]_{rs}=0.
  \end{equation*}
  In particular, under the assumptions in (ii), we have $M_{rs}\geq A_{rs}$ with $M_{rs}=A_{rs}$ if and only if $v_r$ and $v_s$ do not have a connection through $\Vi$. Hence in case (i) and (ii) we have $M\in\mathcal A(G_{\outerVertex})$. Moreover, in case (ii) we have $M_{rs}>0$ for all $(r,s)\in E_{\outerVertex}$.
  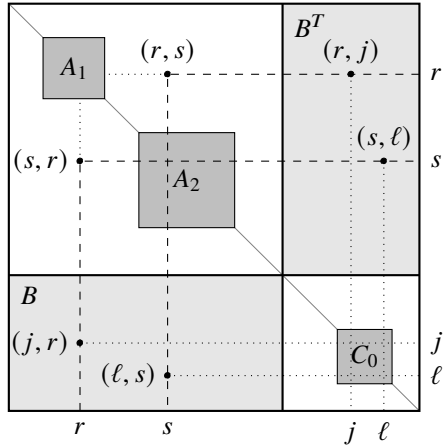
\begin{figure}
    \centering
    \begin{tikzpicture}[scale=0.9,xscale=-1,
        font=\footnotesize,
        ablock/.style={draw,fill=lightgray},
        cblock/.style={draw,ablock},
        bblock/.style={fill=black!10},
        row/.style={fill=black!10},
        transposed/.style={x={(0,1cm)},y={(1cm,0)}},
        mentry/.style={radius=1.25pt,fill,dashed},
        declare function={
            c=2;%
            c1=0.4;%
            c1w=0.8;%
            a1=0.7;%
            a1w=1.4;%
            a2=a1+a1w+0.5;%
            a2w=0.9;%
            a=4;%
            r1=c1+0.75*c1w;%
            r2=c+a2+0.4*a2w;%
            s1=c1+0.15*c1w;%
            s2=c+a1+0.7*a1w;%
          }]
      \draw[help lines] (0,0) -- (a+c,a+c);%

      \path[bblock] (0,c) rectangle ++(c,a);%
      \path[bblock,transposed] (0,c) rectangle ++(c,a);%
      \node[below right] at (c,a+c) {$B^T$};%
      \node[below right,transposed] at (c,a+c) {$B$};%

      \path[cblock] (c1,c1) rectangle ++(c1w,c1w);%
      \node at (c1+c1w/2,c1+c1w/2) {$C_0$};%
      \path[ablock] (c+a1,c+a1) rectangle ++(a1w,a1w);%
      \node at (c+a1+a1w/2,c+a1+a1w/2) {$A_2$};%
      \path[ablock] (c+a2,c+a2) rectangle ++(a2w,a2w);%
      \node at (c+a2+a2w/2,c+a2+a2w/2) {$A_1$};%

      \draw[thick] (0,c) -- (a+c,c);%
      \draw[thick,transposed] (0,c) -- (a+c,c);%

      \draw[dashed] (0,r2)  node[right] {$r$} -| (s2,0) node[below] {$s$};%
      \draw[dashed,transposed] (0,r2)  node[below] {$r$} |- (s2,0) node[right] {$s$};%
      \begin{scope}[dotted]
        \draw (r1,r2) -- (r1,0) node[below,black] {$j$};%
        \draw[transposed] (r1,r2) -- (r1,0) node[right,black] {$j$};%
        \draw (s1,s2) -- (s1,0) node[below,black] {$\ell$};%
        \draw[transposed] (s1,s2) -- (s1,0) node[right,black] {$\ell$};%
        \draw (s2,r2) -- (c+a2,r2);%
        \draw[transposed] (s2,r2) -- (c+a2,r2);%
      \end{scope}
      \draw[mentry] (r1,r2) circle node[above] {$(r,j)$};%
      \draw[mentry,transposed] (r1,r2) circle node[left] {$(j,r)$};%
      \draw[mentry] (s1,s2) circle node[above] {$(s,\ell)$};%
      \draw[mentry,transposed] (s1,s2) circle node[left] {$(\ell,s)$};%
      \draw[mentry] (s2,r2) circle node[above] {$(r,s)$};%
      \draw[mentry,transposed] (s2,r2) circle node[left] {$(s,r)$};%

      \draw[thick] (0,0) rectangle (a+c,a+c);%
    \end{tikzpicture}
    \caption{Block decomposition of $A_G$ connecting $v_r$ and $v_s$ through $\Vi$}
    \label{fig:L-blocks}
  \end{figure}
\end{proof}
\begin{remark}
  The above proof shows that the part $B^T(-C)B$ contributes to the entry $M_{rs}$ if there is a connection through $\Vi$. This is indicated by showing the dashed edges in the reduced graphs in Figure~\ref{fig:graph-with-subgraphs-reduced} and~\ref{fig:nw-4-3} as well as the further examples in Section~\ref{sec:examples}. What is not entirely clear is whether each $M_{rs} \neq 0$ for $(r,s) \in E_{\outerVertex}$. If all terms are positive, then Proposition~\ref{prop:properties-for-dtn-operator-with-inner-vertices} shows that this is the case; however, otherwise, there could be cancellations. We note that if $\Vo = V$, then it is always the case that $M_{rs} \neq 0$ for all $(r,s) \in E_{\outerVertex}$, as stated in Proposition~\ref{prop:explicit-formula-for-dtn-operator}.
\end{remark}

\section{Limit theorems for the Dirichlet-to-Neumann operator}
\label{sec:limit-theorems}
In this section we develop the essential tools to prove Theorem~\ref{thm:main-result}\ref{thm:main:positive}--\ref{thm:main:evpos}. We start with a key technical lemma.
\begin{lemma}
  \label{lem:ergodic}
  Assume that the family of lengths $(L_{kj})_{(k,j) \in E}$ is linearly independent over the field $\bbQ$. For each $(k,j) \in E$ let $(\gamma_{kj,\ell})_{\ell\in\mathbb N}$ be a sequence in $\bbR$ such that
  \begin{equation}
    \label{eq:ergodic-limit-assumption}
    \lim_{\ell\to\infty}\gamma_{kj,\ell}=\gamma_{kj}\in[-\infty,\infty]
    \qquad\text{and}\qquad
    \lim_{\ell\to\infty}\frac{\gamma_{kj,\ell}}{\ell}=0.
  \end{equation}
  Then there exists a sequence $(\lambda_\ell)_{\ell \in \bbN}$ in $(0,\infty)$ with $\lambda_\ell\to\infty$ as $\ell\to\infty$ such that for every edge $(k,j) \in E$ we have that
  \begin{equation}
    \label{eq:ergodic-limits}
    \lim_{\ell\to\infty}\left[\ell \sin\left(\sqrt{\lambda_\ell} L_{kj}\right)\right]=\gamma_{kj}
    \quad\text{and}\quad
    \lim_{\ell\to\infty}\cos\left(\sqrt{\lambda_\ell} L_{kj}\right)=1.
  \end{equation}
  Moreover, if $\gamma_{kj}=0$, then we can choose the $\lambda_\ell$ such that either $\sin(\sqrt{\lambda_\ell}L_{kj})>0$ or $\sin(\sqrt{\lambda_\ell}L_{kj})<0$ for all $\ell\in\mathbb N$.
\end{lemma}
\begin{proof}
  By \eqref{eq:ergodic-limit-assumption} we can choose $\ell_0\in\mathbb N$ such that
  \begin{equation*}
    \frac{|\gamma_{kj,\ell}|}{\ell}<1
  \end{equation*}
  for all $\ell>\ell_0$. By discarding finitely many elements of the sequence and renumbering we can assume that $\ell_0=0$. For $\ell\in\mathbb N$ consider the numbers
  \begin{equation*}
    z_{kj,\ell}:=\sqrt{1-\frac{\gamma_{kj}^2}{\ell^2}}+i\frac{\gamma_{kj}}{\ell}\in\mathbb C.
  \end{equation*}
  Then $z_{kj,\ell}\in\mathbb T:=\{z\in \mathbb C\colon |z|=1\}\subseteq\mathbb C$. Let $(k,j)\in E$. If $\gamma_{kj}>0$ let
  \begin{equation*}
    U_{kj,\ell}:=\left\{z\in\mathbb T\colon\left\lvert z-z_{kj,\ell}\right\rvert<\frac{1}{\ell^2}\text{ and }\repart(z)>0\right\}
  \end{equation*}
  and if $\gamma_{kj}<0$ let
  \begin{equation*}
    U_{kj,\ell}:=\left\{z\in\mathbb T\colon\left\lvert z-z_{kj,\ell}\right\rvert<\frac{1}{\ell^2}\text{ and }\repart(z)<0\right\}.
  \end{equation*}
  Due to that choice the neighbourhood $U_{kj,\ell}$ of $z_{kj,\ell}$ lies in the same quadrant as $z_{kj}$. If $\gamma_{kj}=0$ we can make either of the above choices for all $\ell\in\mathbb N$. For every $\ell\in\mathbb N$ we consider the open set
  \begin{equation*}
    U_\ell:=\prod_{(k,j)\in E} U_{kj,\ell}
  \end{equation*}
  on the $n$-dimensional torus $\mathbb T^n$. By assumption, $L=(L_{kj})_{(k,j)\in E}$ is linearly independent over $\mathbb Q$. As $\mathbb R$ is an infinite dimensional vector space over $\mathbb Q$, we can choose $s\in\mathbb R\setminus\mathbb Q$ such that the family $L\cup\{s^{-1}\}$ is still linearly independent and $sL_{kj}\in\mathbb R\setminus\mathbb Q$ for all $(k,j)\in E$. Then also $(sL_{kj})_{(k,j)\in E}\cup\{1\}$ is linearly independent over $\mathbb Q$. It follows from Kronecker's Lemma that there exists a sequence $(\mu_\ell)_{\ell\in\mathbb N}$ with $\mu_\ell\to\infty$ as $\ell\to\infty$ and
  \begin{equation*}
    e^{\mu_\ell sL_{kj}}\in U_{kj,\ell}
  \end{equation*}
  for all $(k,j)\in E$ and all $\ell\in\mathbb N$, see for instance \cite[Theorem~2.39]{eisner:15:ota}. If we set $\lambda_\ell:=(\mu_\ell s)^2$, then
  \begin{equation*}
    \left\lvert\sin\left(\sqrt{\lambda_\ell}L_{kj}\right)-\impart(z_{kj,\ell})\right\rvert
    =\left\lvert\sin\left(\sqrt{\lambda_\ell}L_{kj}\right)-\frac{\gamma_{kj,\ell}}{\ell}\right\rvert
    \leq\frac{1}{\ell^2}
  \end{equation*}
  and thus
  \begin{equation*}
    \gamma_{kj,\ell}-\frac{1}{\ell}
    <\ell\sin\left(\sqrt{\lambda_\ell}L_{kj}\right)
    <\gamma_{kj,\ell}+\frac{1}{\ell}
  \end{equation*}
  for all $\ell\in\mathbb N$. Hence the first limit in \eqref{eq:ergodic-limits} follows. Similarly, using that
  \begin{equation*}
    \repart(z_{kj,\ell})=\sqrt{1-\dfrac{\gamma_{kj,\ell}^2}{\ell^2}}>0
  \end{equation*}
  and the first limit in \eqref{eq:ergodic-limit-assumption} we deduce that
  \begin{align*}
    \left\lvert\cos\left(\sqrt{\lambda_\ell}L_{kj}\right)-1\right\rvert
     & \leq\left\lvert\cos\left(\sqrt{\lambda_\ell}L_{kj}\right)-\sqrt{1-\frac{\gamma_{kj,\ell}^2}{\ell^2}}\right\rvert
    +\left\lvert\sqrt{1-\frac{\gamma_{kj,\ell}^2}{\ell^2}}-1\right\rvert                                                \\
     & <\frac{1}{\ell^2}+\left\lvert\sqrt{1-\frac{\gamma_{kj,\ell}^2}{\ell^2}}-1\right\rvert\to 0
  \end{align*}
  as $\ell\to\infty$ for all $(k,j)\in E$, establishing the second limit in \eqref{eq:ergodic-limits}. This completes the proof of the lemma.
\end{proof}
In order to prove Theorem~\ref{thm:main-result} we make suitable choices for the $\gamma_{kj}$ in Lemma~\ref{lem:ergodic}. In particular, we use the lemma to show that a rescaled version of $-D_{\lambda,\Vo}$ converges to the \emph{graph Laplacian} $L_G$ for a suitable sequence $(\lambda_\ell)_{\ell\in\mathbb N}$ diverging to infinity. The graph Laplacian is the difference of the \emph{adjacency matrix} $A_G$ given in \eqref{eq:adjacency-matrix} and the diagonal matrix $D_G$ whose $j$-th diagonal entry is the degree of the vertex $v_j$, that is, the number of edges meeting at $v_j$. Hence the $n\times n$ matrix $L_G=A_G-D_G$ has constant row and column sums with value zero.
\begin{corollary}
  \label{cor:ergodic}
  Assume that the family of lengths $(L_{kj})_{(k,j) \in E}$ is linearly independent over the field $\bbQ$. For each $(k,j) \in E$ let $\gamma_{kj} \in [-\infty,\infty]\setminus\{0\}$. Let $D_{\lambda,V}$ be the Dirichlet-to-Neumann operator given by \eqref{eq:dtn-operator}. Then there exists a sequence $(\lambda_\ell)_{\ell\in\mathbb N}$ with $\lambda_\ell\to\infty$ such that
  \begin{equation}
    \label{eq:D-limit}
    \lim_{\ell\to\infty}\frac{1}{\ell\sqrt{\lambda_\ell}}D_{\lambda_\ell,V}=Q,
  \end{equation}
  where $Q$ is the matrix given by the entries
  \begin{equation}
    \label{eq:Q-off-diagonal}
    Q_{kj}=Q_{jk}:=
    \begin{cases}
      -\dfrac{1}{\gamma_{kj}} & \text{if }(k,j)\in E     \\
      0                       & \text{if }(k,j)\not\in E \\
    \end{cases}
  \end{equation}
  if $k\neq j$ and
  \begin{equation}
    \label{eq:Q-diagonal}
    Q_{kk}=-\sum_{\substack{j=1\\j\neq k}}^nQ_{kj}
  \end{equation}
  for $k=1,\dots,n$. By convention we set $1/\pm\infty:=0$. If $\gamma_{kj}>0$ for all $(k,j)\in E$, then $Q$ has spectral bound $\spb(-Q)=0$. If $\gamma_{kj}=1$ for all $(k,j)\in E$, then the matrix $-Q=L_G$ is the graph Laplacian.
\end{corollary}
\begin{proof}
  If $\gamma_{kj}\in\mathbb R$ we let $(\gamma_{kj,\ell})_{\ell\in\mathbb N}$ be the constant sequence with $\gamma_{kj,\ell}=\gamma_{kj}$ for all $\ell\in\mathbb N$. If $\gamma_{kj}=\infty$ we set $\gamma_{kj}:=\sqrt{\ell}$ and if $\gamma_{kj}=-\infty$ we set $\gamma_{kj,\ell}:=-\sqrt{\ell}$ for all $\ell\in\mathbb N$. Then the assumptions of Lemma~\ref{lem:ergodic} are satisfied. Hence there exists a sequence $(\lambda_\ell)_{\ell\in\mathbb N}$ in $(0,\infty$ with $\lambda_\ell\to\infty$ and such that \eqref{eq:ergodic-limits} hold.

  We now recall that the coefficients of $D_{\lambda,V}$ have the form \eqref{eq:dtn-operator-entries}, where $\beta_{kj}(\lambda)$ and $\alpha_{k,j}(\lambda)$ are given by \eqref{eq:beta} and \eqref{eq:alpha}. Hence, if $(k,j)\in E$, then by \eqref{eq:ergodic-limits}
  \begin{equation*}
    \frac{1}{\ell\sqrt{\lambda_\ell}}[D_{\lambda_\ell,V}]_{kj}
    =-\frac{1}{\ell\sin\left(\sqrt{\lambda_\ell}L_{kj}\right)}
    \to-\frac{1}{\gamma_{kj}}=Q_{kj}
  \end{equation*}
  as $\ell\to\infty$. Similarly, using \eqref{eq:ergodic-limits}
  \begin{equation*}
    \frac{1}{\ell\sqrt{\lambda_\ell}}[D_{\lambda_\ell,V}]_{kk}
    =\sum_{\substack{j=1\\j\neq k}}^n\frac{\cos\left(\sqrt{\lambda_\ell}L_{kj}\right)}{\ell\sin\left(\sqrt{\lambda_\ell}L_{kj}\right)}
    \to\sum_{\substack{j=1\\j\neq k}}^n\frac{1}{\gamma_{kj}}=\sum_{\substack{j=1\\j\neq k}}^nQ_{kj}=Q_{kk}
  \end{equation*}
  for every $k=1,\dots,n$ as $\ell\to\infty$.

  By \eqref{eq:Q-diagonal}, the vector $u=[1,\dots,1]^T$ is a positive eigenvector of $-Q$ corresponding to the eigenvalue $0$. Hence, using the Gershgorin disc theorem we see that $\spb(-Q)=0$, see for instance \cite[Theorem~6.1.1]{horn:13:man}.

  Assume now that $\gamma_{kj}=1$ for all $(k,j)\in E$. Then the off-diagonal entries of $-Q$ are given by those of the adjacency matrix of the graph $G$. By \eqref{eq:Q-diagonal} the diagonal elements $Q_{kk}$ are degree of the vertices, showing that $-Q=L_G$.
\end{proof}

\section{Positivity properties of the Dirichlet-to-Neumann semigroup}
\label{sec:positivity}
In this section we focus our attention to the structure of the matrix $D_{\lambda,\Vo}$, in particular when $\Vo$ is a proper subset of $V$. In that case we saw that $-D_{\lambda,\Vo}$ is given by the Schur complement \eqref{eq:DN-Vo}, where the matrices $A$, $B$ and $C$ are blocks of the matrix $-D_{\lambda,V}$ as in \eqref{eq:DN-V}. In particular, the non-trivial off-diagonal entries of $D_{\lambda,V}$ are those with $(k,j)\in E$.

We next deal with the case $\lambda<\lambda_1 (-\Delta_{\Vo})$, where $\lambda_1(-\Delta_{\Vo})$ is given by~\eqref{eq:lambda-1}. This proves Theorem~\ref{thm:main-result}\ref{thm:main:basic}. The proof is inspired by that of \cite[Theorem~4.1]{arendt:12:fei}, but simplified. No special assumptions on the lengths $(L_e)_{e\in E}$ is required for that.
\begin{proposition}
  \label{prop:positivity-less-spb}
  Let $G$ be a connected quantum graph as in Section~\ref{sec:introduction} and let $\Vo$ and $\Vi$ the set of outer and inner vertices. If $\lambda < \lambda_1(-\Delta_{\Vo})$, then the matrix semigroup $(e^{-tD_{\lambda, \Vo}})_{t \ge 0}$ is strongly positive.
\end{proposition}
\begin{proof}
  We have to show that if $\lambda<\lambda_1(-\Delta_{\Vo})$, then all off-diagonal elements of $-D_{\lambda,\Vo}$ are non-negative. Suppose that $\lambda\not\in\sigma(-\Delta_{\Vo})$ and let $u_k$ be the unique solution of \eqref{eq:boundary-value-problem} with $x=v_k\in\Vo$ which exists by Proposition~\ref{prop:boundary-value-problem-and-eigenvalue-problem}. It follows from \eqref{eq:dtn-bilinear-form} that for every $w\in H_{\Vo}^1(G)$ and $k\neq j$ we have that
  \begin{align*}
    [D_{\lambda,\Vo}]_{kj}
     & = (D_{\lambda,\Vo}(v_k),v_j)_{\bbC^{\Vo}} \\
     & = a_\lambda(u_k,u_j)
    = a_\lambda(u_k-w,u_j-w)-a_\lambda(w,w).
  \end{align*}
  By \eqref{eq:lambda-1-rayleigh}
  \begin{equation*}
    a_\lambda(w,w)
    = a(w,w)-\lambda\|w\|^2
    \geq \left(\lambda_1(-\Delta_{\Vo})-\lambda\right)\|w\|_{L^2(G)}^2
  \end{equation*}
  for all $w\in H_{\Vo}^1(G)$ and thus
  \begin{equation}
    \label{eq:off-diagonal-dtn-matrix}
    [D_{\lambda,\Vo}]_{kj}\leq\inf_{w\in H_{\Vo}^1(G)}\left[a_\lambda(u_k-w,u_j-w)+\left(\lambda-\lambda_1(-\Delta_{\Vo})\right)\|w\|_{L^2(G)}^2\right].
  \end{equation}
  As $k\neq j$ and $v_k,v_j\geq 0$ we have that $w_{kj}:=\min(u_k,u_j)\in H_{\Vo}^1(G)$, see \cite[Lemma~7.6]{gilbarg:01:epd}. Then $u_k-w_{kj}$ and $u_j-w_{kj}$ are zero on complementary sets and thus
  \begin{equation*}
    a_\lambda(u_k-w_{kj},u_j-w_{kj})=0.
  \end{equation*}
  Using \eqref{eq:off-diagonal-dtn-matrix} we deduce that if $\lambda<\lambda_1(-\Delta_{\Vo})$, then
  \begin{equation*}
    [D_{\lambda,\Vo}]_{kj}
    \leq \left(\lambda-\lambda_1(-\Delta_{\Vo})\right)\|w_{kj}\|_{L^2(G)}^2
    < 0
  \end{equation*}
  whenever $k\neq j$. As $G_{\outerVertex}$ is connected, $-D_{\lambda,\Vo}$ is irreducible and thus the semigroup generated by $-D_{\lambda,\Vo}$ is strongly positive by Lemma~\ref{lem:positive-semigroups}.
\end{proof}
We next prove Theorem~\ref{thm:main-result}\ref{thm:main:tree} for trees. The idea is that for the semigroup to be eventually positive, at least one off-diagonal entry of the generator needs to be negative. However, it also needs to have sufficiently many positive off-diagonal entries. We use a result from \cite{johnson:22:tsp} that this can never be achieved if the reduced graph $G_{\outerVertex}$ is a tree.
\begin{proposition}
  Let $G$ be a connected graph with set of outer vertices $\Vo$. Suppose that the reduced graph $G_{\outerVertex}$ is a tree. Then $(e^{-D_{\lambda,\Vo}})_{t\geq 0}$ is either positive or not eventually strongly positive.
\end{proposition}
\begin{proof}
  Let $\lambda\in\mathbb R\setminus\sigma(-\Delta_{\Vo})$. We know from Proposition~\ref{prop:properties-for-dtn-operator-with-inner-vertices} that the potentially non-trivial off-diagonal entries of $-D_{\lambda,\Vo}$ are exactly those of the adjacency matrix of $G_{\outerVertex}$. All other off-diagonal entries are zero. For the semigroup to be eventually positive without being positive, we need at least one negative off-diagonal entry, that is, two by the symmetry. Since removing a positive entry is equivalent to removing an edge from the graph, it follows that the positive part of $\omega I-D_{\lambda,\Vo}$ cannot have the pattern determined by the entries of the adjacency matrix of a tree. It follows from \cite[Theorem~1]{johnson:22:tsp} that $\omega I-D_{\lambda,\Vo}$ cannot be eventually positive, that is, there does not exist $k\in\mathbb N$ such that $(\omega I-D_{\lambda,\Vo})^{-1}\gg 0$. Hence, according to Lemma~\ref{lem:eventually-positive-semigroups}(i) the semigroup $(e^{-D_{\lambda,\Vo}})_{t\geq 0}$ cannot be eventually strongly positive without being positive, completing the proof of the proposition.
\end{proof}

Now we consider the more delicate case when $\lambda > \lambda_1(-\Delta_{\Vo})$. We complete the proof of Theorem~\ref{thm:main-result} in a sequence of propositions. We start by proving part~\ref{thm:main:positive}.
\begin{proposition}
  \label{prop:main-ii}
  Let the assumptions of Theorem~\ref{thm:main-result}\ref{thm:main:positive} be satisfied. Then for every $\hat\lambda\in\mathbb R$ there exists $\lambda>\hat\lambda$ such that $(e^{-D_{\lambda,\Vo}})_{t\geq 0}$ is strongly positive.
\end{proposition}
\begin{proof}
  By Corollary~\ref{cor:ergodic} there exists a sequence $(\lambda_\ell)_{\ell\in\mathbb N}$ such that
  \begin{equation*}
    -L_G=\lim_{\ell\to\infty}\frac{1}{\ell\sqrt{\lambda_\ell}}D_{\lambda_\ell,V}.
  \end{equation*}
  As $G$ is connected, the matrix $L_G$ is irreducible by Lemma~\ref{lem:positive-semigroups}. Hence there exists $\ell_0\in\mathbb N$ such that $-D_{\lambda_\ell,V}$ is irreducible with non-negative off-diagonal elements for all $\ell\geq\ell_0$. Therefore $(e^{-tD_{\lambda_\ell,V}})_{t>0}$ is strongly positive for all $\ell\geq\ell_0$ again by Lemma~\ref{lem:positive-semigroups}.

  Now let $\Vo$ be a proper subset of $V$. Consider the block decomposition of $L_G$ of the form \eqref{eq:DN-V}. Without loss of generality we renumber the vertices and edges so that $A$ and $C$ have the form \eqref{eq:A-C-block} with each block irreducible. We note that the row and column sums of $L_G$ are zero. As the graph $G$ is connected, the row sums in each of the blocks $C_j$ are less than or equal to zero, with at least one of them negative. The Gershgorin disk theorem implies that $\spb(C_j)<0$ for each of the blocks. By Proposition~\ref{prop:properties-for-dtn-operator-with-inner-vertices}, the matrix $M$ given by \eqref{eq:M} is irreducible with non-negative off-diagonal elements. As the spectral bound is a continuous function of the matrix entries, there exists $\ell_0\in\mathbb N$ such that the same is true for $-D_{\lambda_\ell,\Vo}$ for all $\ell\geq\ell_0$. Therefore $(e^{-tD_{\lambda_\ell,\Vo}})_{t>0}$ is strongly positive for all $\ell\geq\ell_0$ by Lemma~\ref{lem:positive-semigroups}.
\end{proof}
The next proposition proves Theorem~\ref{thm:main-result}\ref{thm:main:nonpos}.
\begin{proposition}
  \label{prop:main-iii}
  Assume that $|\Vo| \ge 2$. Then, for each $\hat \lambda \in \bbR$, there exists $\lambda > \hat \lambda$ such that the semigroup $(e^{-tD_{\lambda, \Vo}})_{t \ge 0}$ is not eventually positive and thus, in particular, not positive.
\end{proposition}
\begin{proof}
  Choosing $\gamma_{k,j}=-1$ in Corollary~\ref{cor:ergodic} for all $(k,j)\in E$ there exists a sequence $(\lambda_\ell)_{\ell\in\mathbb N}$ such that
  \begin{equation*}
    L_G=\lim_{\ell\to\infty}\frac{1}{\ell\sqrt{\lambda_\ell}}D_{\lambda_\ell,V}.
  \end{equation*}
  As $G$ is connected $L_G$ is irreducible. If we proceed as in the proof of Proposition~\ref{prop:main-ii} with $-D_{\lambda_\ell,\Vo}$ replaced by $D_{\lambda,\Vo}$, then we conclude that $D_{\lambda,\Vo}$ is irreducible with non-negative off-diagonal elements for $\ell$ large enough. Hence $(e^{-tD_{\lambda_\ell,\Vo}})_{t>0}$ cannot be a positive or eventually positive semigroup by Lemma~\ref{lem:positive-matrix-group} if $|\Vo|\geq 2$.
\end{proof}
We finally prove Theorem~\ref{thm:main-result}\ref{thm:main:evpos}.
\begin{proposition}
  \label{prop:main-iv}
  Assume the reduced graph $G_{\outerVertex}$ has a cycle. Then, for each $\hat \lambda \in \bbR$, there exists $\lambda > \hat \lambda$ such that the semigroup $(e^{-tD_{\lambda, \Vo}})_{t \ge 0}$ is eventually strongly positive, but not positive.
\end{proposition}
\begin{proof}
  Suppose $e_{\outerVertex}=(r,s)$ is an edge in a cycle of $G_{\outerVertex}$. Let $\tilde G:=G_{\outerVertex}- e_{\outerVertex}$. Since $e_{\outerVertex} $ is part of a cycle, $\tilde G$ is still connected. Hence, the corresponding graph Laplacian $L_{\tilde G}$ is irreducible by Lemma~\ref{lem:positive-semigroups}. Consider the block decomposition of $L_{\tilde G}$ of the form \eqref{eq:DN-V}, that is,
  \begin{equation*}
    L_{\tilde G}=
    \begin{bmatrix}
      \tilde A & \tilde B^T \\
      \tilde B & \tilde C
    \end{bmatrix}
    .
  \end{equation*}
  As argued in the proof of Proposition~\ref{prop:main-ii}, $\spb(\tilde C)<0$ and $[\tilde B^T(-\tilde C)^{-1}\tilde B]_{rs}\geq 0$. Given $\varepsilon>0$ we set
  \begin{equation*}
    \gamma_{rs}^{(\varepsilon)}:= -\dfrac{1}{\varepsilon+[\tilde B^T(-\tilde C)^{-1}\tilde B]_{rs}}
  \end{equation*}
  If $(r,s)\neq (k,j)\in E$ we set $\gamma_{kj}=1$. By Corollary~\ref{cor:ergodic} there exists a sequence $(\lambda_\ell)_{\ell\in\mathbb N}$ such that
  \begin{equation*}
    Q^{(\varepsilon)}=-\lim_{\ell\to\infty}\frac{1}{\ell\sqrt{\lambda_\ell}}D_{\lambda_\ell,V}
  \end{equation*}
  with
  \begin{equation*}
    [Q^{(\varepsilon)}]_{kj}=
    \begin{cases}
      -\varepsilon-[\tilde B^T(-\tilde C)^{-1}\tilde B]_{rs} & \text{if }(k,j)=(r,s)                         \\
      [L_{\tilde G}]_{kj}                                    & \text{if }k\neq j\text{ and }(k,j)\neq (r,s).
    \end{cases}
  \end{equation*}
  Moreover, by Corollary~\ref{cor:ergodic} the diagonal elements of $Q^{(\varepsilon)}$ are such that the row and column sums of $Q^{(\varepsilon)}$ are zero. We furthermore note that the off-diagonal entries of $Q^{(\varepsilon)}$ and $L_{\tilde G}$ coincide except for the entry $(r,s)$. In particular, if we consider the block decomposition of $Q^{(\varepsilon)}$ of the form \eqref{eq:DN-V}, that is,
  \begin{equation*}
    Q^{(\varepsilon)}=
    \begin{bmatrix}
      A_\varepsilon & B_\varepsilon^T \\B_\varepsilon&C_\varepsilon
    \end{bmatrix}
    ,
  \end{equation*}
  then $B_\varepsilon=\tilde B$ and $C_\varepsilon=\tilde C$ for all $\varepsilon>0$. Hence,
  \begin{equation*}
    Q^{(\varepsilon)}=
    \begin{bmatrix}
      A_\varepsilon & \tilde B^T \\\tilde B&\tilde C
    \end{bmatrix}
  \end{equation*}
  for all $\varepsilon>0$. By choice of $\gamma_{rs}$ we have that
  \begin{align*}
    [A_\varepsilon+B_\varepsilon^T(-C_\varepsilon)B_\varepsilon]_{rs} & =\frac{1}{\gamma_{rs}}+[B^T(-C)^{-1}B]_{rs}                                                                    \\
                                                                      & =-\varepsilon-[\tilde B^T(-\tilde C)^{-1}\tilde B]_{rs}+[\tilde B^T(-\tilde C)^{-1}\tilde B]_{rs}=-\varepsilon
  \end{align*}
  and therefore
  \begin{equation*}
    Q_{\Vo}^{(\varepsilon)}=-\lim_{\ell\to\infty}\frac{1}{\ell\sqrt{\lambda_\ell}}D_{\lambda_\ell,\Vo}=A_\varepsilon+B^T(-C)B
  \end{equation*}
  with $[Q_{\Vo}^{(\varepsilon)}]_{rs}=-\varepsilon$ and $\sign[Q_{\Vo}^{(\varepsilon)}]_{kj}=\sign[L_{G_{\Vo}}]_{kj}$ for all $(k,j)$ with $k\neq j$, $(k,j)\neq(r,s)$ and all $\varepsilon>0$. Clearly,
  \begin{equation*}
    Q_{\Vo}^{(0)}:=\lim_{\varepsilon\to 0}Q_{\Vo}^{(\varepsilon)}
  \end{equation*}
  has non-negative off-diagonal elements. As $\tilde G$ is connected, $L_{\tilde G}$ is irreducible. Hence, also $Q_{\Vo}^{(0)}$ is irreducible and thus generates a strongly positive semigroup $(e^{tQ_{\Vo}})_{t\geq 0}$. The set of generators of strongly eventually positive matrix semigroups is open as shown in Lemma~\ref{lem:eventually-positive-semigroups}(ii) and thus $(e^{tQ_{\Vo}(\varepsilon)})_{t>0}$ is eventually strongly positive but not positive for all $\varepsilon$ small enough. If we fix such $\varepsilon>$, applying Lemma~\ref{lem:eventually-positive-semigroups}(ii) again it follows that for large enough $\ell$ the semigroup $(e^{-tD_{\lambda_\ell,\Vo}})_{t\geq 0}$ is eventually strongly positive, but not positive.
\end{proof}
We finally give a proof of Theorem~\ref{thm:main-commensurable}.
\begin{proof}[Proof of Theorem~\ref{thm:main-commensurable}]
  By assumption, there exists $L>0$ and a family $(n_e)_{e\in E}$ in $\bbN$ such that $L_e=n_eL$ for all $e\in E$. Hence, if $\lambda$ is as in \eqref{eq:main-commensurable}, then
  \begin{align*}
    \sin\left(\sqrt{\lambda}L_e\right)
     & =\sin\left(\sqrt{\lambda}n_eL\right)
    =\sin\left(\left(\sqrt{\mu}+\frac{2\pi p}{L}\right)n_eL\right) \\
     & =\sin\left(\sqrt{\mu}L_e+2\pi pn_e\right)
    =\sin\left(\sqrt{\mu}L_e\right)
  \end{align*}
  and similarly for $\cos\left(\sqrt{\lambda}L_e\right)$. Taking into account \eqref{eq:alpha}, \eqref{eq:beta} and \eqref{eq:DN-Vo} it follows that
  \begin{equation*}
    \frac{1}{\sqrt{\lambda}}D_{\lambda,\Vo}=\frac{1}{\sqrt{\mu}}D_{\mu,\Vo}
  \end{equation*}
  whenever $\lambda=\left(\sqrt{\mu}+\frac{2\pi p}{L}\right)^2$ with $\mu\in\left(0,\lambda_1(-\Delta_{\Vo})\right)$ and $p\in\bbN$. As $-D_{\lambda,\Vo}$ generates a strongly positive semigroup by Proposition~\ref{prop:positivity-less-spb}, the same is true for $-D_{\lambda,\Vo}$ if $\lambda$ is given as in \eqref{eq:main-commensurable}.
\end{proof}
We note that for suitable configurations, also eventual strong positivity is possible for arbitrarily large positive $\lambda$, but that is more delicate to prove. The reason is that some off-diagonal entries need to be negative, but small enough relative to the other positive off-diagonal entries. Hence, to establish such a result one one typically takes the generator of a strongly positive semigroup with a zero off-diagonal entry that can be perturbed so that it becomes negative.

\appendix

\section{Appendix: Positive and eventually positive matrix semigroups}
\label{sec:matrix-semigroups}
The purpose of this appendix is to collect some facts on positive and eventually positive matrix semigroups. We recall that a matrix $M$ is \emph{reducible} if there exists a permutation matrix $P$ such that we have a block decomposition of the form
\begin{equation*}
  PMP^T=
  \begin{bmatrix}
    M_{11} & 0 \\M_{21}&M_{22}
  \end{bmatrix}.
\end{equation*}
If no such permutation matrix exists we call $M$ \emph{irreducible}. Recall that a directed graph is strongly connected if there is a path between any pair of vertices. Our graphs are not directed, but the convention is that every edge represents two edges, one in each direction, see for instance \cite[Sections~1.1 and~1.3]{berkolaiko:13:iqg}. This makes the adjacency matrix of an undirected graph symmetric. Here is a characterisation of irreducible matrices in terms of graphs, see for instance \cite[Theorem~2.2.7]{berman:94:nmm}.
\begin{lemma}[Irreducible matrices]
  \label{lem:irreducible-matrix}
  A matrix is irreducible if and only if its non-zero off-diagonal entries correspond to those of the adjacency matrix of a strongly connected graph.
\end{lemma}

In line with the theory of positive operators we call a matrix \emph{positive} if all entries are \emph{non-negative} and write $M\geq 0$. We call $M$ \emph{strongly positive} if all entries are strictly greater than zero, and write $M\gg 0$. The semigroup $(e^{tM})_{t\geq 0}$ is called \emph{eventually strongly positive} if there exists $t_0\geq 0$ such that $e^{tM}\gg 0$ for all $t>t_0$. It is called \emph{strongly positive} if it is eventually strongly positive with $t_0=0$. We will frequently use the following well known facts.
\begin{lemma}[Strongly positive semigroups]
  \label{lem:positive-semigroups}
  Let $M$ be a real $n\times n$ matrix. Then the following assertions are equivalent.
  \begin{enumerate}[label={\normalfont(\roman*)}]
  \item $(e^{tM})_{t\geq 0}$ is strongly positive.
  \item $M$ is irreducible and there exists $\omega\in\mathbb R$ with $\omega I+M\geq 0$
  \item There exists $\omega\in\mathbb R$ with $\omega I+M\geq 0$ and the positive off-diagonal entries of $M$ determine the adjacency matrix of a strongly connected graph.
  \item There exists $\omega\in\mathbb R$ and $k\in\mathbb N$ such that $\omega I+M\geq 0$ and $(\mu I+M)^k\gg 0$.
  \item $(\mu I-M)^{-1}\gg 0$ for all $\mu>\spb(M)$.
  \end{enumerate}
\end{lemma}
\begin{proof}
  Let $\omega\in\mathbb R$ such that $\omega I + M\geq 0$. The equivalence of (ii)--(iv) now follows from \cite[Theorem~2.1.3 and~2.2.7]{berman:94:nmm}. It follows from
  \begin{equation*}
    e^{tM}=e^{-\omega t}e^{t(\omega I+M)}=e^{-\omega t}\sum_{k=0}^\infty\frac{t^k}{k!}(\omega I+M)^k,
  \end{equation*}
  that (vi) implies (i) and (i) implies (ii). It follows from a Neumann series expansion that
  \begin{equation*}
    (\mu I-M)^{-1}=\sum_{k=1}^\infty\frac{1}{(\mu+\omega)^{k+1}}(\omega I+ M)^k
  \end{equation*}
  whenever $\mu>\spb(M)$ and thus (vi) implies (v) and (v) implies (ii).
\end{proof}

The next lemma collects some facts about eventually strongly positive matrix semigroups, see \cite[Theorem~3.3]{noutsos:08:rhn} and \cite[Proposition~4.6]{daners:18:tpt}.

\begin{lemma}[Eventually strongly positive semigroups]
  \label{lem:eventually-positive-semigroups}
  Let $M$ be a real $n\times n$ matrix. Then the following assertions hold.
  \begin{enumerate}[label={\normalfont(\roman*)}]
  \item $(e^{tM})_{t\geq 0}$ is eventually strongly positive if and only if there exists $\mu\in\mathbb R$ and $k\in\mathbb N$ such that $(\mu I+M)^k\gg 0$.
  \item The set of real matrices that generate an eventually strongly positive semigroup is open in $\mathbb R^{n\times n}$.
  \end{enumerate}
\end{lemma}
The next lemma characterises positive matrix groups.
\begin{lemma}
  \label{lem:positive-matrix-group}
  Let $M \in \bbC^{d \times d}$.
  \begin{enumerate}[label={\normalfont (\roman*)}]
  \item If the semigroup $(e^{tM})_{t \ge 0}$ is positive and if the semigroup $(e^{-tM})_{t \ge 0}$ is eventually positive, then $e^{tM}$ is positive for all $t \in \bbR$.
  \item If $e^{tM}$ is positive for all $t \in \bbR$, then $M$ is a diagonal matrix with only real entries.
  \end{enumerate}
\end{lemma}
\begin{proof}
  (i) There exists a number $t_0 > 0$ such that $e^{tM}$ is positive for all $t \in \bbR \setminus (-t_0,0)$. If $t \in (-t_0,0)$, then $t-t_0,t_0\not\in(-t_0,0)$, so we have $e^{tM} = e^{(t-t_0)M}e^{t_0M} \ge 0$.

  (ii) Since $e^{tM} \geq 0$ for all $t \geq 0$, every off-diagonal entry of $M$ must be non-negative. Since also $e^{t(-M)} \geq 0$ for all $t \geq 0$, every off-diagonal entry of $-M$ is also non-negative, and so they are all $0$. Hence $M$ is a diagonal matrix. That $M$ is real follows from the fact that $M = \left.\left(\frac{d}{dt}e^{tM}\right)\right\rvert_{t=0}$ and $e^{tM}$ is real (since non-negative).
\end{proof}
It is worthwhile pointing out that assertion~(ii) of the above lemma has some infinite dimensional analogues; see \cite[Proposition~3.5]{keicher:06:psb}, \cite[Corollary~2.4]{wolff:08:tpp} or \cite[Theorem~2.2]{gerlach:19:cpo}.

\paragraph*{Acknowledgements}
JBK would like to express his thanks to the University of Sydney for its hospitality during a visit in which part of this project was completed, and DD likewise to the University of Wuppertal for a very pleasant visit. The work of JBK was partly supported by the Funda\c{c}\~ao para a Ci\^encia e a Tecnologia, Portugal, via the research centers GFM (grants \href{https://doi.org/10.54499/UIDB/00208/2020}{UIDB/00208/2020} and  \href{https://doi.org/10.54499/UIDP/00208/2020}{UIDP/00208/2020}) and CIDMA (grants \href{https://doi.org/10.54499/UIDB/04106/2020}{UIDB/04106/2020} and \href{https://doi.org/10.54499/UIDP/04106/2020}{UIDP/04106/2020}).

\pdfbookmark[1]{\refname}{biblio}%
\bibliographystyle{doi}%
\bibliography{literature}

\end{document}